\documentclass[a4paper]{amsart}

\newtheorem{theorem}{Theorem}[section]
\newtheorem{lemma}[theorem]{Lemma}
\newtheorem{cor}[theorem]{Corollary}

\theoremstyle{definition}
\newtheorem{definition}[theorem]{Definition}

\theoremstyle{remark}
\newtheorem{remark}[theorem]{Remark}

\numberwithin{equation}{section}
\newcommand{\Q}{\mathbb{Q}}
\newcommand{\C}{\mathbb{C}}
\newcommand{\N}{\mathbb{N}}
\newcommand{\R}{\mathbb{R}}

\DeclareMathOperator{\rank}{rank}

\DeclareMathOperator{\ind}{ind}

\DeclareMathOperator{\ch}{ch}
\DeclareMathOperator{\coker}{coker}

\title{Dirac-Operators and symmetries of quasitoric manifolds}

\author{Michael Wiemeler}
\address{School of Mathematics,
    Alan Turing Building, The University of Manchester, Oxford Road, Manchester M13 9PL, UK}
\curraddr{Max-Planck-Institute for Mathematics, Vivatsgasse 7, D-53111 Bonn, Germany}
\email{wiemeler@mpim-bonn.mpg.de}
\thanks{The research was supported by SNF Grant No. PBFRP2-133466.}



\begin{document}
\begin{abstract}
We establish a vanishing result for indices of certain twisted Dirac operators on \(\text{Spin}^c\)-manifolds with non-abelian Lie-group actions.
We apply this result to study non-abelian symmetries of quasitoric manifolds.
We give upper bounds for the degree of symmetry of these manifolds.
\end{abstract}

\maketitle


\section{Introduction}
\label{intro}

A quasitoric manifold is a \(2n\)-dimensional manifold with a well-behaved action of an \(n\)-dimensional torus such that the orbit space is an \(n\)-dimensional simple polytope.
Quasitoric manifolds were introduced by Davis and Januszkiewicz \cite{0733.52006} as topological generalizations of non-singular projective toric varieties.

In this paper we study the degree of symmetry of quasitoric manifolds and give upper bounds in various situations.
For example we show that \(\C P^n\) is the most symmetric \(2n\)-dimensional quasitoric manifold.
Moreover, we construct infinitely many quasitoric manifolds of dimension \(2n=4k\), \(k>0\), which do not admit an action of a semi-simple compact connected Lie-group.

For a smooth manifold \(M\), the degree of symmetry \(N(M)\) of \(M\) is defined to be the maximum of the dimensions of those compact Lie-groups which act smoothly and effectively on \(M\).

Similarly one defines the semi-simple symmetry degree \(N^{ss}(M)\) of \(M\) as
\begin{align*}
  N^{ss}(M)&=\max\{\dim G;\; G \text{ compact semi-simple Lie-group, }\\
&\quad\quad G \text{ acts smoothly and effectively on } M\}
\end{align*}
and the torus symmetry degree \(T(M)\) of \(M\) to be
the maximum of the dimensions of those compact tori which act smoothly and effectively on \(M\).

It is well known that, for an \(n\)-dimensional manifold \(M\), \(N(M)\leq \frac{n(n+1)}{2}\) with equality holding if and only if \(M=S^n\) or \(M=\R P^n\).
Moreover, we have \(T(M)\leq n\) with equality holding if and only if \(M\) is a torus.
If \(\chi(M)\neq 0\), then we have \(T(M)\leq \frac{n}{2}\).

A quasitoric manifold has positive Euler-characteristic.
Therefore the torus symmetry degree of a quasitoric manifold is maximal in the class of manifolds with non-vanishing Euler-characteristic.

In this paper we show that \(\C P^n\) has maximal degree of symmetry among the quasitoric manifolds of dimension \(2n\), i.e. \(N(M)<N(\C P^n)=n^2+2n\) for all quasitoric manifolds \(M\neq \C P^n\) with \(\dim M = 2n\) (see Theorem \ref{sec:highly-symm-quas-7}).

Moreover, we generalize a vanishing result for indices of certain twisted Dirac operators on \(\text{Spin}^c\)-manifolds with \(\text{Pin}(2)\)-action found by Dessai \cite{0963.19002} to manifolds with actions of more general groups (see Theorem~\ref{sec:twist-dirac-oper-8}).
This generalization allows us to prove that if a \(2n\)-dimensional \(\text{Spin}^c\)-manifold \(M\) with \(\chi(M)\neq 0\) admits such a twisted Dirac operator with non-vanishing index then its degree of symmetry is bounded from above by \(3n\) with equality holding if and only if \(M=\prod_{i=1}^n S^2\) (see Corollary~\ref{sec:twist-dirac-oper-11}).
We show that a \(2n\)-dimensional quasitoric manifold whose orbit polytope admits a facet coloring with \(n\) colors is an example of such a manifold.
Hence, we get:

\begin{theorem}[Corollary \ref{sec:twist-dirac-oper-16}]
\label{sec:introduction-1}
If \(M\) is a \(2n\)-dimensional quasitoric manifold whose orbit polytope admits a facet coloring with \(n\) colors, then we have \(N(M)\leq 3n\) with equality holding if and only if \(M=\prod_{i=1}^n S^2\).  
\end{theorem}

Moreover, we show that if a \(2n\)-dimensional \(\text{Spin}^c\)-manifold \(M\)  admits a twisted Dirac-operator with non-vanishing index and an effective action of a non-abelian compact connected Lie-group \(G\), then the order of the Weyl-group of \(G\) divides the Euler-characteristic of \(M\) (see Corollary \ref{sec:twist-dirac-oper-7}).
This enables us to prove the following result.

\begin{theorem}[Corollary~\ref{sec:twist-dirac-oper-15}, Corollary~\ref{sec:twist-dirac-oper-13}]
  Let \(n\geq 2\). Then we have:
  \begin{enumerate}
  \item If \(n\) is odd, then there are infinitely many quasitoric manifolds \(M\) of dimension \(2n\) with \(N^{ss}(M)\leq 3\), i.e. the only semi-simple simply connected compact Lie-group which can act almost effectively on \(M\) is \(SU(2)\).
  \item If \(n\) is even, then there are infinitely many quasitoric manifolds of dimension \(2n\) on which no semi-simple compact connected Lie-group can act effectively. 
  \end{enumerate}
\end{theorem}

We also study those \(2n\)-dimensional quasitoric manifolds whose orbit polytopes admit facet colorings with \(n\) colors and have relatively many non-abelian  symmetries.
For these manifolds we have the following theorem.

\begin{theorem}[Theorem~\ref{sec:quas-manif-with-1}, Theorem~\ref{sec:quas-manif-with-2}]
  Let \(M\) be a \(2n\)-dimensional quasitoric manifold whose orbit polytope admits a facet coloring with \(n\) colors. Assume that one of the following two conditions holds:
  \begin{enumerate}
  \item There is an action of a compact Lie-group on \(M\) such that \(\dim M/G \leq 1\).
  \item We have \(N(M)\geq 3n-4\).
  \end{enumerate}
Then \(M\) is the total space of a fiber bundle over \(\prod S^2\).
\end{theorem}

By considering twisted Dirac-operators we can also prove the following theorem:
\begin{theorem}[Corollary \ref{sec:two-vanish-results-2}]
\label{sec:introduction}
  Let \(M\) be a \(\text{Spin}\)-manifold with \(p_1(M)=0\), \(G\) an exceptional Lie-group or \(G=\text{Spin}(2l+1)\) or \(G=Sp(l)\) with \(l=1,3,6\) or \(l\geq 15\) and \(T\) a maximal torus of \(G\). If the Witten-genus of \(M\) does not vanish, then we have \(N^{ss}(M\times \prod_{i=1}^k G/T)=k\dim G\).
\end{theorem}

If more generally \(G\) is a semi-simple compact connected Lie-group with maximal torus \(T\), then we still get upper bounds for the semi-simple symmetry degree of \(M\times G/T\).
But we do not get the exact value of \(N^{ss}(M\times G/T)\) in the more general setting.
It should be noted here, that it has been shown by Hauschild \cite{0623.57024} that  the semi-simple symmetry degree of \(G/T\) is equal to \(\dim G\) if \(G\) is a semi-simple compact connected Lie-group with maximal torus \(T\).
So Theorem~\ref{sec:introduction} may be viewed as a partial generalization of his result.

This paper is organized as follows.
In sections \ref{sec:twisted} and \ref{sec:prod} we discuss indices of twisted Dirac-operators on \(\text{Spin}^c\)-manifolds.
Then we prove Theorem~\ref{sec:introduction} in section \ref{sec:van_witten}.
In section \ref{sec:qt_mfd}  we apply the results of the previous sections to show that there are quasitoric manifolds with low semi-simple symmetry degree.
In sections \ref{sec:cohom1} and \ref{sec:cube} we study those quasitoric manifolds which have a non-vanishing index and have relatively many non-abelian symmetries.
In section \ref{sec:highly} we show that \(\C P^n\) is the most symmetric quasitoric manifold in dimension \(2n\). This section is independent of the other sections.
In appendix \ref{sec:tori} we prove some technical details which are needed in section \ref{sec:twisted}.

I would like to thank Anand Dessai for comments on an earlier version of this paper.
I would also like to thank Nigel Ray and the University of Manchester for hospitality while I was working on this paper.

\section{Twisted Dirac-operators and non-abelian Lie-group actions}
\label{sec:twisted}

The purpose of this section is to generalize some results of \cite{0963.19002} to a class of non-abelian compact non-connected Lie-groups.

We begin with a review of some well known facts about \(\text{Spin}^c\)-manifolds and the results of \cite{0963.19002} and \cite{0944.58019}.
For more background information about the group \(\text{Spin}^c(k)\) and \(\text{Spin}^c\)-structures on manifolds see for example \cite{0146.19001}, \cite{0247.57010} and \cite{0395.57020}.

An orientable manifold \(M\) has a \(\text{Spin}^c\)-structure if and only if the second Stiefel-Whitney-class \(w_2(M)\)  of \(M\) is the \(\!\!\!\!\mod 2\)-reduction of an integral class \(c\in H^2(M;\mathbb{Z})\).
Associated to a \(\text{Spin}^c\)-structure on \(M\) there is a complex line bundle.
We denote the first Chern-class of this line bundle by \(c_1^c(M)\).
Its \(\!\!\!\!\mod 2\)-reduction is \(w_2(M)\) and we should note that any integral cohomology class whose \(\!\!\!\!\mod 2\)-reduction is \(w_2(M)\) may be realized as the first Chern-class of a line bundle associated to some \(\text{Spin}^c\)-structure on \(M\).

Let \(M\) be a \(2n\)-dimensional \(\text{Spin}^c\)-manifold on which \(S^1\) acts smoothly.
We say that the \(S^1\)-action on \(M\) lifts into the \(\text{Spin}^c\)-structure \(P\), if there is a \(S^1\)-action on \(P\) which commutes with the \(\text{Spin}^c(2n)\)-action on \(P\) such that the projection \(P\rightarrow M\) is \(S^1\)-equivariant.

\begin{lemma}
\label{sec:twist-dirac-oper-12}
  The \(S^1\)-action on \(M\) lifts into the \(\text{Spin}^c\)-structure if and only if it lifts to an action on the line bundle associated to the \(\text{Spin}^c\)-structure.
\end{lemma}
\begin{proof}
  If the \(S^1\)-action lifts to an action on the \(\text{Spin}^c\)-structure \(P\) of \(M\), then it also lifts into the associated line bundle \(P\times_{\text{Spin}^c} \C\).

Now assume that the \(S^1\)-action on \(M\) lifts into the associated line bundle of \(P\).
Let \(Q\) be the oriented orthogonal frame bundle of \(M\).
Then the \(S^1\)-action lifts into \(Q\).
Moreover, by \cite[p. 127-128]{0247.57010}, the action on \(M\) lifts into \(P\) if and only if the action on \(Q\) lifts into the \(S^1\)-bundle
\begin{equation*}
  \xi:P\rightarrow P/S^1=Q.
\end{equation*}

Now we consider the Serre-spectral sequence for the fibration \(Q\rightarrow Q_{S^1}\rightarrow BS^1\).
By Corollary 1.3 of \cite[p. 14]{0346.57014}, the \(S^1\)-action lifts into \(\xi\) if and only if
\begin{align*}
  d_2c_1(\xi)&=0&\text{and}&&d_3c_1(\xi)=0.
\end{align*}
Because \(H^*(BS^1;\mathbb{Z})\) is concentrated in even degrees, this holds if and only if \(d_2c_1(\xi)=0\).

Let \(\xi'\) be the \(S^1\)-bundle over \(Q\) associated to the pullback of the line bundle associated to \(P\).
Then the \(S^1\)-action on \(Q\) lifts into \(\xi'\).
Since \(\xi'=\xi^2\), we have \(2d_2c_1(\xi)=d_2c_1(\xi')=0\).
Because \(E_2^{2,1}=H^2(BS^1;H^1(Q;\mathbb{Z}))\) is torsion-free, it follows that the \(S^1\)-action lifts into \(P\).
\end{proof}

If the \(S^1\)-action on \(M\) lifts into the \(\text{Spin}^c\)-structure, then we have an \(S^1\)-equivariant \(\text{Spin}^c\)-Dirac operator \(\partial_c\).
Its \(S^1\)-equivariant index is an element of the representation ring of \(S^1\) and is defined as
\begin{equation*}
  \ind_{S^1}(\partial_c) = \ker \partial_c - \coker \partial_c \in R(S^1).
\end{equation*}

Let \(V\) be a \(S^1\)-equivariant complex vector bundle over \(M\) and \(W\) an even-dimensional \(S^1\)-equivariant \(\text{Spin}\) vector bundle over \(M\).
With this data we build a power series \(R\in K_{S^1}(M)[[q]]\) defined by
\begin{align*}
  R&= \bigotimes_{k=1}^\infty S_{q^k}(\tilde{TM}\otimes_\R \C)\otimes \Lambda_{-1}(V^*)\otimes \bigotimes_{k=1}^\infty \Lambda_{-q^k}(\tilde{V}\otimes_\R \C)\\& \otimes \Delta(W)\otimes\bigotimes_{k=1}^\infty \Lambda_{q^n}(\tilde{W}\otimes_\R\C).
\end{align*}
Here \(q\) is a formal variable, \(\tilde{E}\) denotes the reduced vector bundle \(E -\dim E\), \(\Delta(W)\) is the full complex spinor bundle associated to the \(\text{Spin}\)-vector bundle \(W\), and \(\Lambda_t\) (resp. \(S_t\)) denotes the exterior (resp. symmetric) power operation. The tensor products are, if not indicated otherwise, taken over the complex numbers.

After extending \(\ind_{S^1}\) to power series we may define:

\begin{definition}
  Let \(\varphi^c(M;V,W)_{S^1}\) be the \(S^1\)-equivariant index of the \(\text{Spin}^c\)-Dirac operator twisted with \(R\):
  \begin{equation*}
    \varphi^c(M;V,W)_{S^1}= \ind_{S^1}(\partial_c \otimes R)\in R(S^1)[[q]].
  \end{equation*}
We denote by \(\varphi^c(M;V,W)\) the non-equivariant version of this index:
  \begin{equation*}
    \varphi^c(M;V,W)= \ind(\partial_c \otimes R)\in \mathbb{Z}[[q]].
  \end{equation*}
\end{definition}

The Atiyah-Singer index theorem \cite{0164.24301} allows us to calculate 
\begin{equation*}
  \varphi^c(M;V,W)=\langle e^{c_1^c(M)/2}\ch(R)\hat{A}(M),[M]\rangle.
\end{equation*}
Here we have
\begin{equation*}
  \ch(R)=Q_1(TM)Q_2(V)Q_3(W)
\end{equation*}
with
\begin{align*}
  Q_1(TM)&=\ch(\bigotimes_{k=1}^{\infty} S_{q^k}(\tilde{TM}\otimes_\R \C))=\prod_i\prod_{k=1}^\infty \frac{(1-q^k)^2}{(1-e^{x_i}q^k)(1-e^{-x_i}q^k)},\\
  Q_2(V)&=\ch( \Lambda_{-1}(V^*)\otimes \bigotimes_{k=1}^\infty \Lambda_{-q^k}(\tilde{V}\otimes_\R \C))\\ &= \prod_i (1-e^{-v_i})\prod_{k=1}^{\infty} \frac{(1-e^{v_i}q^k)(1-e^{-v_i}q^k)}{(1-q^k)^2},\\
  Q_3(W)&=\ch(\Delta(W)\otimes\bigotimes_{k=1}^\infty \Lambda_{q^n}(\tilde{W}\otimes_\R\C))\\ &=\prod_i(e^{w_i/2}+e^{-w_i/2})\prod_{k=1}^{\infty} \frac{(1+e^{w_i}q^k)(1+e^{-w_i}q^k)}{(1+q^k)^2}.
\end{align*}
Here \(\pm x_i\) (resp. \(v_i\) and \(\pm w_i\)) denote the formal roots of \(TM\) (resp. \(V\) and \(W\)).
If \(c_1^c(M)=c_1(V)\), then we have
\begin{equation*}
  e^{c_1^c(M)/2}Q_2(V)= e(V)\frac{1}{\hat{A}(V)}\prod_i\prod_{k=1}^{\infty} \frac{(1-e^{v_i}q^k)(1-e^{-v_i}q^k)}{(1-q^k)^2}=e(V)Q_2'(V).
\end{equation*}

Note that if \(M\) is a \(\text{Spin}\)-manifold,
then there is a canonical \(\text{Spin}^c\)-structure on \(M\).
With this \(\text{Spin}^c\)-structure \(\varphi^c(M;0,TM)\) is equal to the elliptic genus of \(M\).
Moreover, \(\varphi^c(M;0,0)\) is the Witten-genus of \(M\).

Dessai proved the following theorem.
\begin{theorem}[{\cite[Theorem 3.2, p. 243]{0944.58019}}]
\label{sec:twist-dirac-oper-9}
  Assume that the equivariant Pontrjagin-class \(p_1^{S^1}(V+W-TM)\) restricted to \(M^{S^1}\) is equal to \(\pi_{S^1}^*(Ix^2)\) modulo torsion, where \(\pi_{S^1}:BS^1\times M^{S^1} \rightarrow BS^1\) is the projection on the first factor, \(x\in H^2(BS^1;\mathbb{Z})\) is a generator and \(I\) is an integer.
Assume, moreover, that \(c_1^c(M)\) and \(c_1(V)\) are equal modulo torsion.
If \(I<0\), then \(\varphi^c(M;V,W)_{S^1}\) vanishes identically. 
\end{theorem}

Let \(G\) be a compact Lie-group such that:
\begin{enumerate}
\item \label{item:1} There is an exact sequence of Lie-groups
  \begin{equation*}
    1 \rightarrow T \rightarrow G\rightarrow W(G)\rightarrow 1,
  \end{equation*}
where \(T\) is a torus and \(W(G)\) a finite group.
\item\label{item:2} If condition (\ref{item:1}) holds, then \(G\) acts by conjugation on \(T\).
Since \(T\) is abelian this action factors through \(W(G)\). We assume that this action of \(W(G)\) is non-trivial on \(T\).
\end{enumerate}

An action of \(G\) on a manifold \(M\) is called nice if \(G\) acts almost effectively on \(M\) and if the induced action on \(H^*(M;\mathbb{Z})\) is trivial.

For nice \(G\)-actions on \(\text{Spin}^c\)-manifolds we have the following vanishing result.

\begin{theorem}
\label{sec:twist-dirac-oper-8}
  Let \(M\) be a \(\text{Spin}^c\)-manifold on which \(G\) acts nicely such that \(M^G\neq \emptyset\).
  Let \(V\) and \(W\) be sums of complex line bundles over \(M\) such that \(W\) is \(\text{Spin}\), \(c_1(V)=c_1^c(M)\) modulo torsion and \(p_1(V+W-TM)=0\) modulo torsion. 
Assume that \(b_1(M)=0\) or that the \(G\)-action on \(M\) extends to an action of a simply connected compact Lie-group. 
Then \(\varphi^c(M;V,W)\) vanishes.
\end{theorem}

\begin{remark}
  Theorem \ref{sec:twist-dirac-oper-8} is a generalization of Theorem 4.4 of \cite[p. 521]{0963.19002}.
\end{remark}

Before we prove Theorem~\ref{sec:twist-dirac-oper-8} we state three lemmas about the equivariant cohomology of \(G\)-manifolds which are needed in the proof.

\begin{lemma}
\label{sec:twist-dirac-oper}
  Let \(M\) be a nice \(G\)-manifold such that \(M^G\neq \emptyset\) and \(b_1(M)=0\). Then 
  \begin{equation*}
    0 \rightarrow H^2(BG;\mathbb{Z})\rightarrow H^2_G(M;\mathbb{Z})\rightarrow H^2(M;\mathbb{Z})\rightarrow 0
  \end{equation*}
is exact.
\end{lemma}
\begin{proof}
  We consider the Serre-spectral sequence for the fibration \(M\rightarrow M_G\rightarrow BG\). 
Because the \(G\)-action on \(M\) is nice we have
\begin{equation*}
  E_2^{p,q}=H^p(BG;H^q(M;\mathbb{Z})).
\end{equation*}
Since \(b_1(M)=0\), we have \(E_\infty^{1,1}=0\). 
Therefore we have an exact sequence
\begin{equation*}
  0\rightarrow E_\infty^{2,0}\rightarrow H^2_G(M;\mathbb{Z})\rightarrow E_\infty^{0,2}\rightarrow 0.
\end{equation*}
Because \(M^{G}\neq \emptyset\), \(H^*(BG;\mathbb{Z})\rightarrow H_G^*(M;\mathbb{Z})\) is injective.
Hence, we have 
\begin{equation*}
H^*(BG;\mathbb{Z})=E_2^{*,0}=E_\infty^{*,0}  
\end{equation*}
 and the differentials \(d_r:E_r^{*-r,r-1}\rightarrow E_r^{*,0}\) vanish.

It remains to show that \(E_\infty^{0,2}=E_2^{0,2}=H^2(M;\mathbb{Z})\).
That is equivalent to \(d_r:E_r^{0,2}\rightarrow E_r^{r,3-r}\) vanishes for all \(r\).

Now we have \(E_2^{2,1}=0\) because \(b_1(M)=0\).
Therefore \(d_2\) vanishes.
Since there are \(G\)-fixed points in \(M\), \(d_3\) vanishes.
The differentials \(d_r\), \(r>3\), vanish because \(E_r^{r,3-r}=0\) for \(r>3\).
Therefore the statement follows.
\end{proof}

\begin{lemma}
\label{sec:twist-dirac-oper-17}
  Let \(M\) be a nice \(G\)-manifold. If the \(G\)-action on \(M\) extends to an action of a simply connected compact Lie-group \(\hat{G}\), then the natural map
  \begin{equation*}
    H_{\hat{G}}^2(M;\mathbb{Z})\rightarrow H^2(M;\mathbb{Z})
  \end{equation*}
is an isomorphism. Moreover, \(H_{{G}}^2(M;\mathbb{Z})\rightarrow H^2(M;\mathbb{Z})\) is surjective.
\end{lemma}
\begin{proof}
  Since \(B\hat{G}\) is three-connected the first statement follows from an inspection of the Serre spectral sequence for the fibration \(M\rightarrow M_{\hat{G}}\rightarrow B\hat{G}\) as in the proof of Lemma~\ref{sec:twist-dirac-oper}.
Then the second statement follows because \(H_{\hat{G}}^2(M;\mathbb{Z})\rightarrow H^2(M;\mathbb{Z})\) factors through \(H_G^2(M;\mathbb{Z})\).
\end{proof}

\begin{lemma}
\label{sec:twist-dirac-oper-1}
Assume that \(T^{W(G)}\) is finite or equivalently that \(\dim (LT)^{W(G)}=0\) or  \(\dim (LT^*)^{W(G)}=0\).
  Let \(M\) be a nice \(G\)-manifold, then
  \begin{equation*}
    H^4(BG;\Q)\rightarrow H_G^4(M;\Q)\rightarrow H^4(M;\Q)
  \end{equation*}
is exact.
\end{lemma}
\begin{proof}
  Because \(\dim (LT^*)^{W(G)}=0\), we have by Proposition 20.4 of \cite[p. 68]{0158.20503} \(H^i(BG;\Q)=0\) for \(i=1,2,3\).
Therefore from the Serre spectral sequence of the fibration \(M\rightarrow M_G\rightarrow BG\) we get an exact sequence
\begin{equation*}
  0\rightarrow E_\infty^{4,0}\rightarrow H_G^4(M;\Q)\rightarrow E_\infty^{0,4}\rightarrow 0.
\end{equation*}
Since \(H^4(BG;\Q)\) surjects to \(E_\infty^{4,0}\) and \(E_\infty^{0,4}\) injects into \(H^4(M;\Q)\), the statement follows.
\end{proof}

Now we are ready to prove Theorem~\ref{sec:twist-dirac-oper-8} in two special cases.
The general case will follow from these special cases.

\begin{lemma}
\label{sec:twist-dirac-oper-2}
Assume that \(T^{W(G)}\) is finite. Then Theorem~\ref{sec:twist-dirac-oper-8} holds.
\end{lemma}
\begin{proof}
  Let \(V=\bigoplus L_i\) and \(W=\bigoplus L_i'\) with \(L_i,L_i'\) line bundles.
  By Lemmas \ref{sec:twist-dirac-oper}, \ref{sec:twist-dirac-oper-17} and Corollary 1.2 of \cite[p. 13]{0346.57014}, the \(G\)-action on \(M\) lifts into each line bundle \(L_i,L_i'\).
  Therefore \(p_1^G(V+W-TM)\) is well defined.
Moreover, by Lemma~\ref{sec:twist-dirac-oper-12}, the action of every \(S^1\subset T\subset G\) lifts into the \(\text{Spin}^c\)-structure on \(M\).

By Theorem~\ref{sec:twist-dirac-oper-9}, it is sufficient to show that, for \(S^1 \hookrightarrow T \hookrightarrow G\),
\begin{equation*}
  p_1^{S^1}(V+W-TM)=\rho(S^1,G)^*p_1^G (V+W-TM) = a \pi_{S^1}^*(x^2),
\end{equation*}
with \(a\in \Q\), \(a<0\), and \(x\in H^2(BS^1;\mathbb{Z})\) a generator.
Here \(\rho(S^1,G)^*: H_G^*(M;\Q)\rightarrow H_{S^1}^*(M;\Q)\) is the map induced by the inclusion \(S^1\hookrightarrow G\) and \(\pi_{S^1}^*:H^*(BS^1;\Q)\rightarrow H_{S^1}^*(M;\Q)\) is the natural map.

By Lemma \ref{sec:twist-dirac-oper-1}, there is an \(\alpha\in H^4(BG;\Q)\) with \(\pi_G^*(\alpha)=p_1^G(V+W-TM)\).
Therefore we have
\begin{equation*}
  p_1^{S^1}(V+W-TM)=\pi_{S^1}^*\rho(S^1,G)^*\alpha = a \pi_{S^1}^*(x^2)
\end{equation*}
with \(a \in\Q\).
It remains to show that \(a< 0\).
We restrict \(p_1^{S^1}(V+W-TM)\) to a \(G\)-fixed point \(y\).
Then we have
\begin{equation*}
  p_1^{S^1}(V+W-TM)|_y=\sum \alpha_i^2 + \sum \beta_i^2 - \sum \gamma_i^2,
\end{equation*}
where  \(\alpha_i\) is the weight of the \(S^1\)-representation on the fiber of \(L_i\) over \(y\),
\(\beta_i\) is the weight of the \(S^1\)-representation on the fiber of \(L_i'\) over \(y\) and the \(\gamma_i\) are the weights of the \(S^1\)-representation \(T_yM\).

The representations on the fibers of \(L_i,L_i'\) are restrictions of one-dimensional \(G\)-representations to \(S^1\).
Because \((LT^*)^{W(G)}=0\), all such representations are trivial.
Therefore \(a = -\sum \gamma_i^2 < 0\) follows, because \(S^1\) acts non-trivially on \(M\).
\end{proof}

\begin{lemma}
\label{sec:twist-dirac-oper-4}
 Assume that \(W(G)\) is cyclic. Then Theorem~\ref{sec:twist-dirac-oper-8} holds.
\end{lemma}
\begin{proof}
  We show that \(G\) has a subgroup satisfying the assumptions of Lemma \ref{sec:twist-dirac-oper-2}.
Then the statement follows from that lemma.

By Lemma~\ref{sec:groups-acting-tori-1}, there are two \(W(G)\)-invariant subtori \(T_1\) and \(T_2\) of \(T\) such that
\begin{itemize}
\item \(W(G)\) acts trivially on \(T_1\).
\item \(T_2^{W(G)}\) is finite.
\item \(T\) is generated by \(T_1\) and \(T_2\).
\end{itemize}
Let \(g\in G\) be a preimage of a generator of \(W(G)\).
Then we have \(g^{\# W(G)}\in T\). Let \(t_1\in T_1\) and \(t_2\in T_2\) such that \(g^{\# W(G)}=t_1t_2\).
Moreover, let \(t\in T_1\) such that \(t_1=t^{\# W(G)}\).
Then \(gt^{-1}\) is another preimage of the generator of \(W(G)\) and \((gt^{-1})^{\# W(G)}\in T_2\).

Let \(G'\) be the subgroup of \(G\) generated by \(gt^{-1}\) and \(T_2\).
Then there is an exact sequence
\begin{equation*}
  1\rightarrow T_2\rightarrow G'\rightarrow W(G)\rightarrow 1.
\end{equation*}
Therefore \(G'\) satisfies the assumptions of Lemma~\ref{sec:twist-dirac-oper-2}.
\end{proof}

In the situation of Theorem \ref{sec:twist-dirac-oper-8}, there is always a cyclic subgroup \(W(H)\) of \(W(G)\) which acts non-trivially on \(T\).
If \(H\) is the preimage of \(W(H)\) under the map \(G\rightarrow W(G)\), then Theorem \ref{sec:twist-dirac-oper-8} follows from Lemma~\ref{sec:twist-dirac-oper-4} applied to the restricted action of \(H\) on \(M\).

From Theorem~\ref{sec:twist-dirac-oper-8} we get the following corollaries about actions of compact connected non-abelian Lie-groups on \(\text{Spin}^c\)-manifolds.

\begin{cor}
\label{sec:twist-dirac-oper-5}
  Let \(G\) be a compact connected non-abelian Lie-group and \(M\) a \(\text{Spin}^c\)-manifold with \(\varphi^c(M;V,W)\neq 0\) and \(V\), \(W\) as in Theorem~\ref{sec:twist-dirac-oper-8}.
  Assume that \(G\) acts almost effectively on \(M\) and that \(T\) is a maximal torus of \(G\). 
Then, for all \(x\in M^T\), \(G_x=T\) holds.
\end{cor}
\begin{proof}
  Let \(\tilde{G}=G'\times T_0\) be a covering group of \(G\) with \(G'\) a semi-simple simply connected compact Lie-group and \(T_0\) a torus.
  Then we have \(\tilde{G}_x=G'_x\times T_0\).
  We will show that \(G'_x\) is a maximal torus of \(G'\).
  From this the statement follows.
  Since, for each compact connected non-abelian Lie-group \(H\), there is a group homomorphism \(\text{Pin}(2)\rightarrow H\) with finite kernel, \(G_x'^0=T'\) is a maximal torus of \(G'\) by Theorem \ref{sec:twist-dirac-oper-8}.

  Assume that \(G'_x\neq T'\).
  Then there is an exact sequence
  \begin{equation*}
    1 \rightarrow T'\rightarrow G'_x \rightarrow G'_x/T'\rightarrow 1
  \end{equation*}
  and we have \(G_x'/T'\subset N_{G'} T'/T'\).
  Therefore \(G'_x/T'\) acts non-trivially on \(T'\).
  But this is a contradiction to Theorem \ref{sec:twist-dirac-oper-8}.
\end{proof}

\begin{cor}
\label{sec:twist-dirac-oper-7}
  Let \(M\) and \(G\) as in Corollary \ref{sec:twist-dirac-oper-5}. Then \(\# W(G) | \chi(M)\).
\end{cor}
\begin{proof}
  We have \(\chi(M)=\chi(M^T)\), where \(T\) is a maximal torus of \(G\).
 By Corollary \ref{sec:twist-dirac-oper-5}, \(W(G)\) acts freely on \(M^T\). Therefore we get
  \begin{equation*}
    \chi(M)=\# W(G) \chi (M^T/W(G)).
  \end{equation*}
\end{proof}

The following two corollaries give upper bounds for the degree of symmetry of a \(\text{Spin}^c\)-manifold which admits a twisted Dirac-operator with non-zero index.

\begin{cor}
\label{sec:twist-dirac-oper-3}  
Let \(M\) be a \(2n\)-dimensional \(\text{Spin}^c\)-manifold with \(\varphi^c(M;V,W)\neq 0\) and \(V\), \(W\) as in Theorem~\ref{sec:twist-dirac-oper-8} and \(G\) be a compact connected Lie-group with
  \begin{enumerate}
  \item \(\dim G -\rank G>2n\) or
  \item \(\dim G- \rank G=2n\) and \(\rank G < T(M)\).
  \end{enumerate}
Then there is no effective action of \(G\) on \(M\).
\end{cor}
\begin{proof}
  Let \(\tilde{G}=G'\times T_0\) be a covering group of \(G\) with \(G'\) a semi-simple simply connected compact Lie-group and \(T_0\) a torus.
  Let \(x\in M\).
 Then by Theorem~\ref{sec:twist-dirac-oper-8} the identity component of \(G'_x\) must be a torus.
  Therefore \(\dim G_x \leq \rank G\).
  Moreover, there is an embedding of \(G/G_x\) in \(M\).
  In case (1) this is impossible.

  In case (2) we have, by dimension reasons, that \(M=G/H\) and \(H\) has maximal rank in \(G\).
  By Corollary~\ref{sec:twist-dirac-oper-5}, \(H\) must be a maximal torus of \(G\).
  Moreover, \(G\) is semi-simple because it acts effectively on \(M\).
  The torus symmetry degree of \(G/H\) was calculated by Hauschild \cite[Theorem 3.3, p. 563]{0623.57024}.
  It is equal to \(\rank G\), which contradicts our assumption that \(\rank G<T(M)\).
\end{proof}

Note that, if \(G\) is a compact Lie-group which acts effectively on a manifold \(M\) as in the above corollary, then the rank of \(G\) is bounded from above by the torus symmetry degree of \(M\). 
Therefore we have \(N(M)\leq 2n + T(M)\).
If the Euler-characteristic of \(M\) is non-zero, we have \(T(M)\leq n\), so that we get the following corollary.

\begin{cor}
\label{sec:twist-dirac-oper-11}
  Let \(M\) be a \(2n\)-dimensional \(\text{Spin}^c\)-manifold with \(\chi(M)\neq 0\) and \(\varphi^c(M;V,W)\neq 0\) with \(V\), \(W\) as in Theorem~\ref{sec:twist-dirac-oper-8}  and \(G\) a compact connected Lie-group which acts effectively on \(M\).
  Then \(\dim G \leq 3n\).
  If \(\dim G= 3n\), then \(M=\prod S^2\).
\end{cor}
\begin{proof}
 By the discussion above, we only have to prove the second statement.

If \(\dim G=3n\), then we must have \(\rank G =n\) and \(M=G/T\), where \(T\) is a maximal torus of \(G\).
Therefore \(G\) is semi-simple.
Because for a simple Lie-group \(G'\) we have \(\dim G' \geq 3\rank G'\) with equality holding if and only if \(G'\) is a quotient of \(SU(2)\), we see that \(G\) has a covering group of the form \(\prod SU(2)\).
Therefore the statement follows.
\end{proof}

\section{Products and connected sums}
\label{sec:prod}

In this section we discuss the calculation of the indices \(\varphi^c(M;V,W)\) for the case where \(M\) is a connected sum or a product of \(\text{Spin}^c\)-manifolds.
The formulas derived here will be used in our applications of the results of the previous section in Sections~\ref{sec:van_witten} and \ref{sec:qt_mfd}.

For cartesian products of \(\text{Spin}^c\)-manifolds we have the following lemma.

\begin{lemma}
\label{sec:prod-conn-sums-2}
  Let \(M_1,M_2\) be even-dimensional \(\text{Spin}^c\)-manifolds, \(V_i\rightarrow M_i\) complex vector bundles and \(W_i\rightarrow M_i\), \(i=1,2\), \(\text{Spin}\) vector bundles.
Then \(M_1\times M_2\) is naturally a \(\text{Spin}^c\)-manifold and 
\begin{equation*}
  \varphi^c(M_1\times M_2;p_1^*V_1\oplus p_2^*V_2,p_1^*W_1\oplus p_2^*W_2)=
 \varphi^c(M_1;V_1,W_1) \varphi^c(M_2;V_2,W_2),
\end{equation*}
where \(p_i:M_1\times M_2\rightarrow M_i\), \(i=1,2\), is the projection.
\end{lemma}
\begin{proof}
  Let \(Q_i\in H^{\dim M_i}(M_i;\Q)[[q]]\) be the degree \(\dim M_i\) part of
  \begin{equation*}
e^{c_1^c(M_i)/2}\ch(R)\hat{A}(M_i)\in H^*(M_i;\Q)[[q]],    
  \end{equation*}
 \(i=1,2\).
Then we have
\begin{align*}
   \varphi^c(M_1\times M_2;p_1^*V_1\oplus p_2^*V_2,p_1^*W_1\oplus p_2^*W_2)&=\langle p_1^*Q_1p_2^*Q_2,[M_1\times M_2]\rangle\\
&=\langle Q_1,[M_1]\rangle\langle Q_2,[M_2]\rangle\\
&= \varphi^c(M_1;V_1,W_1)\varphi^c(M_2;V_2,W_2).
\end{align*}
\end{proof}

The connected sum of two \(\text{Spin}^c\)-manifolds is again a \(\text{Spin}^c\)-manifold.
For these manifolds we have the following lemma.

\begin{lemma}
\label{sec:prod-conn-sums}
  Let \(M_1,M_2\) be \(\text{Spin}^c\)-manifolds of the same even dimension greater or equal to four, \(V_i\rightarrow M_i\), \(i=1,2\), complex vector bundles which are sums of complex line bundles and \(W_i\rightarrow M_i\), \(i=1,2\), \(\text{Spin}\)-bundles which are sums of complex line bundles such that
  \begin{align}
    \label{eq:3}
    c_1(V_i)&=c_1^c(M_i)&\text{and}&& p_1(V_i+W_i-TM_i)&=0.
  \end{align}
Then \(M_1\# M_2\) has a \(\text{Spin}^c\)-structure, such that \(c_1^c(M_1 \# M_2)=c_1^c(M_1)+c_1^c(M_2)\).

If \(\dim V_1> \dim V_2\), then there are vector bundles \(V\rightarrow M_1\# M_2\), \(W\rightarrow M_1\# M_2\) which are sums of complex line bundles satisfying (\ref{eq:3}) such that
\begin{equation*}
   \varphi^c(M_1\# M_2;V,W)=2^{\dim_\C W_2}\varphi^c(M_1;V_1,W_1).
\end{equation*}

If \(\dim V_1=\dim V_2\), then the same holds with
\begin{equation*}
   \varphi^c(M_1\# M_2;V,W)=2^{\dim_\C W_2}\varphi^c(M_1;V_1,W_1)+2^{\dim_\C W_1}\varphi^c(M_2;V_2,W_2).
\end{equation*}
\end{lemma}
\begin{proof}
  Let \(V_i=\bigoplus_{j=1}^{k_i}L_{ji}\) and \(W_i=\bigoplus_{j=1}^{k_i'}L_{ji}'\) for \(i=1,2\).
 Then the \(L_{ji},L_{ji}'\) extend uniquely to vector bundles over \(M_1\# M_2\), such that the restriction to \(M_k \), \(k\neq i\) is trivial.
We denote these extensions also by \(L_{ji},L_{ji}'\).

Let
 \begin{align*}
   V&=\bigoplus_{j=1}^{\max\{k_1,k_2\}}L_{j1}\otimes L_{j2}& &\text{and}&
   W&=\bigoplus_{i=1}^2\bigoplus_{j=1}^{k_i} L_{ji}',
 \end{align*}
where \(L_{ji}\) is the trivial complex line bundle for \(j>k_i\).

The cohomology ring of \(M_1\# M_2\) with coefficients in a ring \(R\) is isomorphic to
\begin{equation}
\label{eq:4}
  H^*(M_1;R)\times H^*(M_2;R)/I,
\end{equation}
where \(I\) is the ideal generated by \((1,-1)\) and \((\xi_1,-\xi_2)\).
Here \(\xi_i\) denotes the orientation class of \(M_i\).
Moreover, for the characteristic classes of \(M_1 \# M_2\) we have
\begin{align*}
  w_i(M_1\# M_2)&=w_i(M_1)+w_i(M_2),&p_i(M_1\# M_2)&= p_i(M_1)+p_i(M_2), & i&>0.
\end{align*}
Therefore there is a \(\text{Spin}^c\)-structure on \(M_1 \# M_2\) with \(c_1^c(M_1 \# M_2)=c_1^c(M_1)+c_1^c(M_2)\).

For the vector bundles \(V\) and \(W\) defined above, we have
\begin{align*}
  c_1(V)&=c_1(V_1)+c_1(V_2)=c_1^c(M_1)+c_1^c(M_2)=c_1^c(M_1\# M_2),\\
  p_1(V)&=\sum_{j=1}^{k_1}c_1(L_{j1})^2 + \sum_{j=1}^{k_2}c_1(L_{j2})^2=p_1(V_1)+p_1(V_2),\\
  p_1(W)&=p_1(W_1)+p_1(W_2).
\end{align*}
Therefore we have \(p_1(V+W-TM_1\# M_2)=0\).

Now we have assuming \(\dim V_1 \geq \dim V_2\),
\begin{align*}
  \varphi^c(M_1\# M_2;V,W)&=\langle e(V) Q_2'(V)Q_3(W)Q_1(TM_1\# M_2)\hat{A}(M_1\# M_2),[M_1\# M_2]\rangle\\
&= \langle e(\bigoplus_{i=k_2+1}^{k_1}L_{j1})( e(\bigoplus_{i=1}^{k_2}L_{j1}) + e(\bigoplus_{i=1}^{k_2} L_{j2}))\\
&\quad\quad Q_2'(V)Q_3(W)Q_1(TM_1\# M_2)\hat{A}(M_1\# M_2),[M_1\# M_2]\rangle.
\end{align*}
It follows from (\ref{eq:4}) that for \(i>0\) the \(i\)th Pontrjagin-class of \(W\) is given by \(p_i(W_1)+p_i(W_2)\). A similar statement holds for the Chern-classes of \(V\).

Since \(2^{-\dim_\C W}Q_2'(V)Q_3(W)Q_1(TM)\hat{A}(M)\) is a power series with constant term one in the Pontrjagin-classes of \(V\), \(W\) and \(TM\) whose coefficients do not depend on \(V\), \(W\) and \(TM\), it follows that, for \(i>0\), 
\begin{align*}
2^{-\dim_\C W}(Q_2'(V)Q_3(W)&Q_1(TM_1\# M_2)\hat{A}(M_1 \# M_2))_i\\ 
&=2^{-\dim_\C W_1}(Q_2'(V_1)Q_3(W_1)Q_1(TM_1)\hat{A}(M_1))_i\\
 &+2^{-\dim_\C W_2}(Q_2'(V_2)Q_3(W_2)Q_1(TM_2)\hat{A}(M_2))_i.  
\end{align*}
  Here \((Q_2'(V)Q_3(W)Q_1(TM))_i\) denotes the degree \(4i\) part of \(Q_2'(V)Q_3(W)Q_1(TM)\).

Now the statement follows from (\ref{eq:4}).
\end{proof}

\section{Two vanishing results for the Witten-genus}
\label{sec:van_witten}

In this section we prove vanishing results for the Witten-genus of a \(\text{Spin}\)-manifold \(M\) with \(p_1(M)=0\) such that a product \(M\times M'\) admits an action of a compact connected semi-simple Lie-group of high rank.

Our first result is as follows.

\begin{theorem}
\label{sec:two-vanish-results}
  Let \(M\) be a \(\text{Spin}\)-manifold such that \(p_1(M)\) is torsion.
  Moreover, let \(M'\) be a \(2n\)-dimensional \(\text{Spin}^c\)-manifold such that there are \(x_1,\dots,x_n\in H^2(M';\mathbb{Z})\) with
  \begin{enumerate}
  \item \(\sum_{i=1}^n x_i=c_1^c(M')\) modulo torsion,
  \item \(\sum_{i=1}^n x_i^2=p_1(M')\) modulo torsion,
  \item\label{item:6} \(\langle \prod_{i=1}^n x_i, [M']\rangle \neq 0\).
  \end{enumerate}
If there is an almost effective action of a  semi-simple simply connected  compact Lie-group \(G\) on \(M\times M'\) such that \(\rank G > \rank \langle x_1,\dots, x_n\rangle\), then the Witten-genus \(\varphi^c(M;0,0)\) of \(M\) vanishes.
\end{theorem}
\begin{proof}
Let \(L_i\), \(i=1,\dots,n\), be the line bundle over \(M'\) with \(c_1(L_i)=x_i\).
By Lemma~\ref{sec:twist-dirac-oper-17}, the natural map \(\iota^*:H^2_G(M\times M';\mathbb{Z})\rightarrow H^2(M\times M';\mathbb{Z})\) is an isomorphism.

Therefore by Corollary 1.2 of \cite[p. 13]{0346.57014} the \(G\)-action on \(M\times M'\) lifts into \(p'^*(L_i)\), \(i=1,\dots,n\).
Here \(p': M\times M'\rightarrow M'\) is the projection.
Moreover, by the above cited corollary and Lemma \ref{sec:twist-dirac-oper-12}, the action of every \(S^1\subset G\) lifts
 into the \(\text{Spin}^c\)-structure on \(M\times M'\) induced by the \(\text{Spin}\)-structure on \(M\) and the \(\text{Spin}^c\)-structure on \(M'\).

By Lemma \ref{sec:prod-conn-sums-2}, we have
\begin{equation*}
  \varphi^c(M\times M';\bigoplus_{i=1}^n p'^*L_i,0)=\varphi^c(M;0,0)\varphi^c(M';\bigoplus_{i=1}^n L_i,0).
\end{equation*}
By condition (\ref{item:6}), we have
\begin{align*}
  \varphi^c(M';\bigoplus_{i=1}^n L_i,0)&=\langle Q_1(TM')\prod_{i=1}^n x_i Q_2'(\bigoplus_{i=1}^nL_i) \hat{A}(M'),[M']\rangle\\
&= \langle \prod_{i=1}^n x_i,[M']\rangle\neq 0.
\end{align*}
Hence, \(\varphi^c(M;0,0)\) vanishes if and only if \(\varphi^c(M\times M';\bigoplus_{i=1}^n p'^*L_i,0)\) vanishes.

Let \(T\) be a maximal torus of \(G\).
If there are no \(T\)-fixed points in \(M\times M'\), then the Lefschetz-fixed point formula implies that this index vanishes.
Therefore we may assume that there is a \(T\)-fixed point \(y\in (M\times M')^T\).

As in the proof of  Lemma \ref{sec:twist-dirac-oper-1} one proves that
\begin{equation*}
  H^4(BG;\Q)\rightarrow H^4_G(M\times M';\Q)\rightarrow H^4(M\times M';\Q)
\end{equation*}
is exact.
Therefore there is an \(v\in H^4(BG;\Q)\) such that \(p_1^T(\bigoplus_{i=1}^n p'^*L_i - T(M\times M')) = \pi_T^*\rho(T,G)^*v\).
By Theorem \ref{sec:twist-dirac-oper-9}, it is sufficient to show that there is a homomorphism \(S^1\hookrightarrow T\) such that \(\rho(S^1,T)^*\rho(T,G)^*v=a x^2\), where \(x\in H^2(BS^1;\mathbb{Z})\) is a generator and \(a\in \mathbb{Z}\), \(a<0\).

We have
\begin{equation*}
  \rho(T,G)^*v=p_1^T(\bigoplus_{i=1}^n p'^*L_i - T(M\times M'))|_y = \sum_{i=1}^n a_i^2 -\sum v_i^2,
\end{equation*}
where the \(a_i\in H^2(BT;\mathbb{Z})\), \(i=1,\dots,n\), are the weights of the \(T\)-representations \(p'^*L_i|_y\) and the \(v_i\in H^2(BT;\mathbb{Z})\) are the weights of the \(T\)-representation \(T_y(M\times M')\).

Since \(\rank T>\rank \langle x_1,\dots,x_m\rangle\) and \(a_i=(\rho(T,G)^*(\iota^*)^{-1}p'^*(x_i))|_y\) for \(i=1,\dots,n\), there is a homomorphism \(S^1\hookrightarrow T\) such that \(\rho(S^1,T)^*a_i=0\) for \(i=1,\dots,n\).

For this \(S^1\) we have \(\rho(S^1,T)^*v=a x^2\) with \(a\in \mathbb{Z}\), \(a<0\), because the \(G\)-action is almost effective.
\end{proof}

We will see later in Lemma~\ref{sec:twist-dirac-oper-6} that those \(2n\)-dimensional quasitoric manifolds whose orbit polytopes admit facet colorings with \(n\) colors are examples of manifolds which satisfy the assumptions on \(M'\) in the above theorem.
Other examples of such manifolds are given by those manifolds whose tangent bundle is isomorphic to a sum of complex line bundles and which have non-zero Euler-characteristic.
In particular, homogeneous spaces of the form \(H/T\) with \(H\) a semi-simple compact connected Lie-group and \(T\) a maximal torus of \(H\) are examples of such manifolds.
Since in this case we have \(b_2(H/T)=\rank H\) we get the following corollary.

\begin{cor}
\label{sec:vanish-result-witt-1}
  Let \(M\) be a \(\text{Spin}\)-manifold with \(p_1(M)=0\) and \(H\) a semi-simple compact connected Lie-group.
  If there is an almost effective action of a semi-simple compact connected Lie-group \(G\) on \(M\times H/T\) such that \(\rank G>\rank H\),
  then the Witten-genus of \(M\) vanishes.
\end{cor}

As an application of Corollary~\ref{sec:vanish-result-witt-1} we give a new proof for a theorem of Hauschild.

\begin{cor}[{\cite[Theorem 9, p. 552]{0623.57024}}]
  Let \(H\) be a semi-simple compact connected Lie-group with maximal torus \(T\). Then we have \(N^{ss}(H/T)=\dim H\).
\end{cor}
\begin{proof}
  Let \(G\) be a semi-simple compact connected Lie-group which acts effectively on \(H/T\).
  Since the tangent bundle of \(H/T\) splits as a sum of complex line bundles and \(\chi(H/T)\neq 0\), there is a twisted Dirac operator with non-vanishing index on \(H/T\).
  Therefore, by the first case in Corollary~\ref{sec:twist-dirac-oper-3}, we have
  \begin{equation*}
    \dim G - \rank G\leq \dim H/T=\dim H - \rank H.
  \end{equation*}
By Corollary~\ref{sec:vanish-result-witt-1} applied in the case \(M=pt\), we see that \(\rank G\leq \rank H\).
Therefore it follows that \(\dim G\leq \dim H\).
Since there is an obvious action of \(H\) on \(H/T\), the statement follows.
\end{proof}

Similarly to Theorem~\ref{sec:two-vanish-results} we can prove the following vanishing result for actions of simple compact connected Lie-groups of high rank.

\begin{theorem}
\label{sec:two-vanish-results-3}
  Let \(M\) be a \(\text{Spin}\)-manifold such that \(p_1(M)\) is torsion.
  Moreover, let \(M'\) be a \(2n\)-dimensional \(\text{Spin}^c\)-manifold such that \(p_1(M')\) is torsion and there are \(x_1,\dots,x_n\in H^2(M';\mathbb{Z})\) and \(1=n_1<n_2<\dots<n_{k+1}=n+1\) with
  \begin{enumerate}
  \item \(\sum_{i=1}^n x_i =c_1^c(M')\) modulo torsion,
  \item \(\sum_{i=n_j}^{n_{j+1}-1}x_i^2\) is torsion, for \(j=1,\dots,k\),
  \item \(\langle \prod_{i=1}^n x_i, [M']\rangle \neq 0\).
  \end{enumerate}
If there is an almost effective action of a  simple simply connected compact Lie-group \(G\) on \(M\times M'\) such that \(\rank G>\rank \langle x_{n_j},\dots,x_{n_{j+1}-1}\rangle\) for all \(j=1,\dots,k\), then the Witten-genus \(\varphi^c(M;0,0)\) of \(M\) vanishes.
\end{theorem}
\begin{proof}
  The relations between the \(G\)-equivariant and non-equivariant cohomology of \(M\times M'\) is as described in the proof of Theorem~\ref{sec:two-vanish-results}.
  We consider the same index \(\varphi^c(M\times M';\bigoplus_{i=1}^n p'^*L_i,0)\) as in the proof of that theorem.
 It vanishes if and only if the Witten-genus of \(M\) vanishes.

Let \(T\) be a maximal torus of \(G\).
We may assume that there is a \(T\)-fixed point \(y\) in \(M\times M'\).

As in the proof of Theorem~\ref{sec:two-vanish-results} one sees that there is an \(v\in H^4(BG;\Q)\) such that \(p_1^T(\bigoplus_{i=1}^n p'^*L_i - T(M\times M'))= \pi_T^*\rho(T,G)^*v\).
By Theorem \ref{sec:twist-dirac-oper-9}, it is sufficient to show that there is a homomorphism \(S^1\hookrightarrow T\) such that \(\rho(S^1,T)^*\rho(T,G)^*v=a x^2\), where \(x\in H^2(BS^1;\mathbb{Z})\) is a generator and \(a\in \mathbb{Z}\), \(a<0\).

We have
\begin{equation*}
  \rho(T,G)^*v=p_1^T(\bigoplus_{i=1}^n p'^*L_i - T(M\times M'))|_y = \sum_{i=1}^n a_i^2 -\sum v_i^2,
\end{equation*}
where the \(a_i\in H^2(BT;\mathbb{Z})\), \(i=1,\dots,n\) are the weights of the \(T\)-representations \(p'^*L_i|_y\) and the \(v_i\in H^2(BT;\mathbb{Z})\) are the weights of the \(T\)-representation \(T_y(M\times M')\).

We will show that the \(a_i\), \(i=1,\dots,n\), vanish.

Let \(1\leq j\leq k\).
Since \(H^4(BG;\Q)\rightarrow H^4_G(M\times M';\Q)\rightarrow H^4(M\times M';\Q)\) is exact, there is an \(v_j'\in H^4(BG;\Q)\) such that \(p_1^T(\bigoplus_{i=n_j}^{n_{j+1}-1}p'^*L_i)=\pi_T^*\rho(T,G)^*v_j'\).
Therefore we have
\begin{equation*}
  \sum_{i=n_j}^{n_{j+1}-1}a_i^2 = p_1^T(\bigoplus_{i=n_j}^{n_{j+1}-1}p'^*L_i)|_y=\rho(T,G)^*v_j'.
\end{equation*}
Therefore \(\sum_{i=n_j}^{n_{j+1}-1}a_i^2\) is invariant under the action of the Weyl-group \(W(G)\) of \(G\) on \(H^4(BT;\mathbb{Q})\).

Because \(\dim T>\rank \langle x_{n_j},\dots,x_{n_{j+1}-1}\rangle\) and \(a_i=(\rho(T,G)^*(\iota^*)^{-1}p'^*(x_i))|_y\),
there is an \(S^1\subset T\) such that \(\rho(S^1,T)^*a_i=0\) for \(i=n_j,\dots,n_{j+1}-1\). 
Since \(\sum_{i=n_j}^{n_{j+1}-1}a_i^2\in H^4(BT;\mathbb{Q})\) is \(W(G)\)-invariant, it follows that, for all \(w\in W(G)\),
\begin{equation*}
   0=\rho(wS^1w^{-1},T)^*\sum_{i=n_j}^{n_{j+1}-1}a_i^2=\sum_{i=n_j}^{n_{j+1}-1}(\rho(wS^1w^{-1},T)^*a_i)^2.
\end{equation*}
Since \(H^*(BS^1;\mathbb{Z})=\mathbb{Z}[x]\), this implies that \(\rho(wS^1w^{-1},T)^*a_i=0\) for all \(i=n_j,\dots,n_{j+1}-1\).
Because \(G\) is simple, there are no non-trivial \(W(G)\)-invariant subtori in \(T\).
Therefore we have \(T=\langle wS^1w^{-1};w\in W(G)\rangle\).
Hence, all \(a_i\in H^2(BT;\mathbb{Z})\), \(i=n_j,\dots,n_{j+1}-1\), \(j=1,\dots,k\), vanish.

Hence, \(\rho(S^1,T)^*\rho(T,G)^*v=a x^2\) with \(a<0\) for all non-trivial homomorphisms \(S^1\hookrightarrow T\).
\end{proof}

Examples of those manifolds which satisfy the assumptions on \(M'\) in the above theorem are manifolds with non-vanishing Euler-characteristic whose tangent bundle is isomorphic to a direct sum of complex line bundles \(L_1,\dots,L_n\) such that there are \(1=n_1<n_2<\dots<n_{k+1}=n+1\) with \(p_1(\bigoplus_{i=n_j}^{n_{j+1}-1}L_i)=0\) for all \(j=1,\dots,k\).

If \(H\) is a simple compact connected Lie-group with maximal torus \(T\), then all Pontrjagin classes of \(H/T\) are torsion. Therefore we get the following corollary.

\begin{cor}
\label{sec:two-vanish-results-1}
  Let \(M\) be a \(\text{Spin}\)-manifold with \(p_1(M)=0\) and \(H_1,\dots,H_k\) be simple compact connected Lie-groups with maximal tori \(T_1,\dots,T_k\).
  If there is an almost effective action of a simple compact connected Lie-group \(G\) on \(M\times \prod_{i=1}^k H_i/T_i\) such that \(\rank G>\rank H_i\) for all \(i=1,\dots, k\),
  then the Witten-genus of \(M\) vanishes.
\end{cor}

The Corollaries~\ref{sec:vanish-result-witt-1} and \ref{sec:two-vanish-results-1} can be used to find an upper bound for the semi-simple symmetry degree of \(M\times H/T\), where \(M\) is a \(\text{Spin}\)-manifold with \(p_1(M)=0\) and non-vanishing Witten-genus and \(H\) is a semi-simple compact Lie-group with maximal torus \(T\).
To give this upper bound we need the following constants.
For \(l\geq 1\) let
\begin{equation*}
  \alpha_l=\max\left\{\frac{\dim G}{\rank G}; \; G \text{ a simple compact Lie-group with } \rank G \leq l\right\}.
\end{equation*}
The values of the \(\alpha_l\)'s are listed in Table \ref{tab:erste}.

\begin{table}
  \centering
  \begin{tabular}{|c|c|c|}
    \(l\)&\(\alpha_l\)& \(G_l\)\\\hline\hline
    \(1\)& \(3\)&\(\text{Spin}(3)\)\\\hline
    \(2\)& \(7\)&\(G_2\)\\\hline
    \(3\)& \(7\)&\(\text{Spin}(7), Sp(3)\)\\\hline
    \(4\)& \(13\)&\(F_4\)\\\hline
    \(5\)& \(13\)& none\\\hline
    \(6\)& \(13\)& \(E_6, \text{Spin}(13), Sp(6)\)\\\hline
    \(7\)& \(19\)& \(E_7\)\\\hline
    \(8\)& \(31\)& \(E_8\)\\\hline
    \(9\leq l\leq 14\)& \(31\)& none\\\hline
    \(l\geq 15\)&\(2l+1\)&\(\text{Spin}(2l+1), Sp(l)\)\\
  \end{tabular}
\caption{The values of \(\alpha_l\) and the simply connected compact simple Lie-groups \(G_l\) of rank \(l\) with \(\dim G_l= \alpha_l\cdot l\).}
\label{tab:erste}
\end{table}

\begin{cor}
\label{sec:two-vanish-results-2}
  Let \(M\) be a \(\text{Spin}\)-manifold with \(p_1(M)=0\), such that the Witten-genus of \(M\) does not vanish and \(H_1,\dots, H_k\) simple compact connected Lie-groups with maximal tori \(T_1,\dots,T_k\).
  Then we have
  \begin{equation*}
    \sum_{i=1}^k \dim H_i \leq N^{ss}(M\times \prod_{i=1}^k H_i/T_i) \leq \alpha_l\sum_{i=1}^k\rank H_i,
  \end{equation*}
where \(l=\max\{\rank H_i;\; i=1,\dots,k\}\).
If all \(H_i\) have the same rank and each \(H_i\) has one of the groups listed in Table~\ref{tab:erste} as a covering group, then equality holds in both inequalities.
\end{cor}
\begin{proof}
  Let \(G\) be a compact simply connected semi-simple Lie-group which acts almost effectively on \(M\times \prod_{i=1}^k H_i/T_i\).
  Then, by Corollary \ref{sec:vanish-result-witt-1}, we have \(\rank G \leq \sum_{i=1}^k \rank H_i\).
  By Corollary \ref{sec:two-vanish-results-1}, all simple factors of \(G\) must have rank smaller or equal to \(l\).
  Therefore we have
  \begin{equation*}
    \dim G\leq \alpha_l\rank G \leq \alpha_l\sum_{i=1}^k \rank H_i.
  \end{equation*}
Hence, \(N^{ss}(M\times \prod_{i=1}^k H_i/T_i)\leq \alpha_l\sum_{i=1}^k \rank H_i\) follows.

  Since there is an obvious \(\prod_{i=1}^k H_i\)-action on \(M\times \prod_{i=1}^k H_i/T_i\), the other inequality follows.

If all \(H_i\) have the same rank \(l\)  and each \(H_i\) has one of the groups listed in Table \ref{tab:erste}  as a covering group, then \(\sum_{i=1}^k\dim H_i=\alpha_l\sum_{i=1}^k \rank H_i\).
Therefore we get equality in this case.
\end{proof}

\begin{remark}
  Our methods to prove Corollary \ref{sec:two-vanish-results-2} break down if we consider the stabilization of \(M\) with a homogeneous space \(H/K\) where \(K\) is a closed subgroup of \(H\) which is not a maximal torus.

If \(K\) has not maximal rank in \(H\), then there is a fixed-point-free torus action on \(M\times H/K\).
Therefore all indices \(\varphi^c(M\times H/K;V,W)\) vanish by the Lefschetz-fixed point formula.
If \(K\) is non-abelian and has maximal rank in \(H\), then all indices \(\varphi^c(M\times H/K;V,W)\) vanish by Corollary~\ref{sec:twist-dirac-oper-5}.

Therefore in both cases the starting point of the proofs of Theorems~\ref{sec:two-vanish-results} and \ref{sec:two-vanish-results-3}, namely the existence of an index \(\varphi^c(M\times M';V,W)\) which vanishes if and only if the Witten-genus of \(M\) vanishes, does not hold in the case \(M'=H/K\).
\end{remark}

\section{Twisted Dirac operators and quasitoric manifolds}
\label{sec:qt_mfd}
In this section we apply the results of the previous sections to the study of quasitoric manifolds.
We begin by recalling the definition of quasitoric manifolds and some of their properties established in \cite{0733.52006} (see also \cite{1012.52021}).

A smooth closed simply connected \(2n\)-dimensional manifold \(M\) with a smooth action of an \(n\)-dimensional torus \(T\) is called quasitoric if the following two conditions are satisfied:
\begin{enumerate}
\item The \(T\)-action on \(M\) is locally isomorphic to the standard action of \(T\) on \(\C^n\).
\item The orbit space \(M/T\) is a simple convex \(n\)-dimensional polytope \(P\).
\end{enumerate}

We denote by \(\mathfrak{F}=\{F_1,\dots,F_m\}\) the set of facets of \(P\). Then for each \(F_i\in\mathfrak{F}\), \(M_i=\pi^{-1}(F_i)\) is a closed connected submanifold of codimension two in \(M\) which is fixed pointwise by a one-dimensional subtorus \(\lambda(F_i)=\lambda(M_i)\) of \(T\).
Here \(\pi:M\rightarrow P\) denotes the orbit map.
These \(M_i\) are called the characteristic submanifolds of \(M\).

The cohomology ring of \(M\) is generated by elements of degree two \(u_1,\dots,u_m\in H^2(M;\mathbb{Z})\) such that
\begin{equation*}
  H^*(M;\mathbb{Z})=\mathbb{Z}[u_1,\dots,u_m]/(I+J),
\end{equation*}
where \(I\) is the ideal generated by 
\begin{equation*}
  \left\{\prod_{j=1}^k u_{i_j};\; \bigcap_{j=1}^k F_{i_j}=\emptyset\right\}
\end{equation*}
and \(J\) is generated by linear relations between the \(u_i\), which depend on the function \(\lambda:\mathfrak{F}\rightarrow \{\text{one-dimensional subtori of } T\}\). It should be noted that each \(u_i\) is the Poincar\'e-dual of \(M_i\).

The stable tangent bundle of \(M\) splits as a sum of complex line bundles \(L_1,\dots,L_m\):
\begin{equation*}
  TM\oplus \R^{2m-2n}\cong \bigoplus_{i=1}^m{L_i},
\end{equation*}
such that \(c_1(L_i)=\pm u_i\).
In particular, a quasitoric manifold has always a stable almost complex structure and therefore a \(\text{Spin}^c\)-structure.

So the results of Section \ref{sec:twisted} might be used to find quasitoric manifolds with only a few non-abelian symmetries.
To do so, we have to find quasitoric manifolds \(M\) which admit vector bundles \(V\rightarrow M\) and \(W\rightarrow M\) which satisfy the assumptions of Theorem~\ref{sec:twist-dirac-oper-8} such that \(\varphi^c(M;V,W)\neq 0\).
In the following we say that \(M\) has a non-vanishing index \(\varphi^c(M;V,W)\) if these assumptions are satisfied.

Now we turn to the construction of such quasitoric manifolds.
Because the stable tangent bundle of a quasitoric manifold \(M\) splits as a sum of line bundles \(\bigoplus_{i=1}^m L_i\) it seems to be natural to consider indices \(\varphi^c(M;V,W)\) with \(V=\bigoplus_{i=1}^k L_i\), \(W=\bigoplus_{i=k+1}^m L_i\) and \(W\) a \(\text{Spin}\)-bundle.
But we have the following result:

\begin{theorem}
\label{sec:twist-dirac-oper-14}
  Let \(M\) be quasitoric. Moreover, let \(M_1,\dots,M_m\) be the characteristic submanifolds of \(M\) and \(L_i\rightarrow M\) the complex line bundles with \(c_1(L_i)=PD(M_i)\). Let
  \begin{align*}
    V&=\bigoplus_{i=1}^kL_i &&\text{and}& W&=\bigoplus_{i=k+1}^m L_i
  \end{align*}
  with \(c_1(V)\equiv c_1(M)\mod 2\) and \(c_1(W)\equiv 0\mod 2\).
  Let \(\partial_c\) be the Dirac-operator for a \(\text{Spin}^c\)-structure on \(M\) with \(c_1^c(M)=c_1(V)\).
  Then \(\varphi^c(M;V,W)=0\).
\end{theorem}
\begin{proof}
We have
  \begin{align*}
    \varphi^c(M;V,W)&= \langle e^{c_1^c(M)/2} \ch(R) \hat{A}(M),[M]\rangle\\
    &=\langle e(V) Q_1(TM)Q_2'(V)Q_3(W) \hat{A}(V\oplus W),[M]\rangle\\
    &=\langle Q_1(V) Q_1(W) Q_2'(V)Q_3(W) \hat{A}(V) \hat{A}(W),[N]\rangle\\
    &=\langle Q_1(W)Q_3(W)\hat{A}(W),[N]\rangle\\
&=2^{m-n}\varphi^c(N;0,TN)=0.
  \end{align*}
  Here \(N\) is the intersection \(\bigcap_{i=1}^k M_i\), which is a quasitoric \text{Spin}-manifold.
  \(N\) can not be a point, since otherwise the first Chern-classes of the summands of \(W\) form a basis of \(H^2(M;\mathbb{Z})\). Therefore \(W\) cannot be \(\text{Spin}\) if \(N\) is a point.
 The elliptic genus \(\varphi^c(N;0,TN)\) of \(N\) vanishes, because there is an odd \(S^1\)-action on \(N\) \cite[p. 317]{0712.57010}.
\end{proof}

So we need another idea to construct quasitoric manifolds \(M\) which have non-trivial indices \(\varphi^c(M;V,W)\).
We will prove that those \(2n\)-dimensional quasitoric manifolds whose orbit polytopes admit facet colorings with \(n\) colors are such examples.
Before we do so we summarize some properties of \(n\)-dimensional polytopes and  facet colorings.

A facet coloring with \(d\) colors of a simple  \(n\)-dimensional polytope \(P\)
 is a map \(f:\mathfrak{F}\rightarrow \{1,\dots,d\}\) such that \(f(F_i)\neq f(F_j)\) whenever \(F_i\cap F_j\neq \emptyset\) and \(F_i\neq F_j\).
Because in each vertex of \(P\) there meet \(n\) facets, one needs at least \(n\) colors to color \(P\).
The following description of simple \(n\)-dimensional polytopes which admit a coloring with \(n\) colors is due to Joswig.

\begin{theorem}[{\cite[Theorem 16, p. 255]{1054.05039}}]
\label{sec:twist-dirac-oper-10}
  Let \(P\) be a simple \(n\)-dimensional polytope. Then the following statements are equivalent:
  \begin{enumerate}
  \item \(P\) is even, i.e. each two-dimensional face of \(P\) has an even number of vertices.
\item The graph which consists out of the vertices and edges of \(P\) is bipartite.
\item\label{item:3} The boundary complex \(\partial P^*\) of the dual polytope of \(P\) is balanced, i.e. there is a non-degenerate simplicial map \(\partial P^*\rightarrow \Delta^{n-1}\). Here \(\Delta^{n-1}\) denotes the \((n-1)\)-dimensional simplex.
\item \(P\) admits a facet coloring with \(n\) colors.
  \end{enumerate}
\end{theorem}

Quasitoric manifolds whose orbit polytopes satisfy condition (\ref{item:3}) in the above theorem were described by Davis and Januszkiewicz \cite[p. 425-426]{0733.52006}. 
They show that this is a very rich class of quasitoric manifolds.
We should note that the \(n\)-dimensional cube admits a facet coloring with \(n\) colors. 
Moreover, a simple polytope belongs to this class if \(\partial P^*\) is the barycentric subdivision of a convex polytope.

Now we construct a non-vanishing index \(\varphi^c(M;V,W)\) on every \(2n\)-dimensional quasitoric manifold \(M\) whose orbit polytope admits a facet coloring with \(n\) colors.

\begin{lemma}
\label{sec:twist-dirac-oper-6}
  Let \(M\) be a quasitoric manifold of dimension \(2n\) over the polytope \(P\).
Assume that \(P\) admits a facet coloring with \(n\) colors.
Then there is a \(\text{Spin}^c\)-structure on \(M\) and a complex vector bundle \(V\) which is a sum of line bundles with \(c_1(V)= c_1^c(M)\) and \(p_1(M)=p_1(V)\), such that \(\varphi^c(M;V,0)\) does not vanish.
\end{lemma}
\begin{proof}
  Let \(f:\mathfrak{F}\rightarrow \{1,\dots,n\}\) be a facet coloring of \(P\) with \(n\) colors.

  Let \(V=\bigoplus_{i=1}^nL_{f^{-1}(i)}\), where \(L_{f^{-1}(i)}\) is the line bundle with \(c_1(L_{f^{-1}(i)})=\sum_{F_j\in f^{-1}(i)}\pm u_j\).
  Then we have \(c_1(V) \equiv c_1(M)\mod 2\) and
  \begin{equation*}
    p_1(V)=\sum_{i=1}^n(\sum_{F_j\in f^{-1}(i)}\pm u_j)^2=\sum_{i=1}^n\sum_{F_j\in f^{-1}(i)} u_j^2=p_1(M).
  \end{equation*}
  Consider a \(\text{Spin}^c\)-structure on \(M\) with \(c_1^c(M)=c_1(V)\) and assume that the associated index \(\varphi^c(M;V,0)\) vanishes.
Then \(\varphi^c(M;V,0)\) may be calculated as:
\begin{align*}
  0&= \langle Q_1(TM) \prod_{i=1}^n c_1(L_{f^{-1}(i)}) Q_2'(V)\hat{A}(M),[M]\rangle\\
  &=\langle\prod_{i=1}^n c_1(L_{f^{-1}(i)}),[M]\rangle.
\end{align*}

Therefore we have
\begin{equation}
  \label{eq:1}
  \prod_{i=1}^n \sum_{F_j\in f^{-1}(i)}\pm u_j =0.
\end{equation}
Since the signs of the \(u_j\) may be changed freely, we get, by considering different \(\text{Spin}^c\)-structures and summing up in equation (\ref{eq:1}):
\begin{equation*}
  \forall (F_{i_1},\dots,F_{i_n})\in f^{-1}(1)\times \dots\times f^{-1}(n) \quad\quad \prod_{j=1}^n u_{i_j}=0.
\end{equation*}
 But there is at least one tuple \((F_{i_1},\dots,F_{i_n})\in \mathfrak{F}\times \dots\times \mathfrak{F}\) such that \(\bigcap_{j=1}^nF_{i_j}\) is a vertex of \(P\).
For this tuple we have \(\prod_{j=1}^n u_{i_j}\neq 0\).
Moreover, because \(f\) is a coloring with \(n\) colors, for each \(k\in\{1,\dots,n\}\) there is exactly one \(F_{i_{j_k}}\) with \(f(F_{i_{j_k}})=k\).
Therefore we get a contradiction.
\end{proof}

As a consequence of Lemma~\ref{sec:twist-dirac-oper-6} and the corollaries at the end of Section~\ref{sec:twisted} we get the following corollaries.

\begin{cor}
\label{sec:twist-dirac-oper-16}
  Let \(M\) be a \(2n\)-dimensional quasitoric manifold. If the orbit polytope of \(M\) admits a facet coloring with \(n\) colors, then we have \(N(M)\leq 3n\) with equality holding if and only if \(M=\prod S^2\).  
\end{cor}
\begin{proof}
  This follows directly from Lemma~\ref{sec:twist-dirac-oper-6} and Corollary~\ref{sec:twist-dirac-oper-11}.
\end{proof}

\begin{cor}
\label{sec:quas-manif-over-1}
  Let \(M\) be a quasitoric manifold over the \(n\)-dimensional cube.
  Then the only simple simply connected compact Lie-groups which can act almost effectively on \(M\) are \(SU(2)\) and \(\text{Spin}(5)\).
\end{cor}
\begin{proof}
  By Lemma~\ref{sec:twist-dirac-oper-6}, there is a twisted Dirac operator on \(M\), whose index does not vanish.
  By Corollary \ref{sec:twist-dirac-oper-7}, the order of the Weyl-group of a simple simply connected compact Lie-group, which acts on \(M\), divides the Euler-characteristic of \(M\).
  Because \(\chi(M)=2^n\) and \(SU(2)\) and \(\text{Spin}(5)\) are the only simple simply connected compact Lie-groups \(G\) with \(\#W(G)|2^n\) \cite[p. 74-84]{0708.17005}, the statement follows.
\end{proof}

In the proof of the next corollary of Lemma~\ref{sec:twist-dirac-oper-6} we construct quasitoric manifolds with low semi-simple symmetry degree.

\begin{cor}
\label{sec:twist-dirac-oper-15}
  In each dimension greater or equal to four, there are infinitely many quasitoric manifolds \(M\) with \(N^{ss}(M)\leq 3\).
\end{cor}
\begin{proof}
  Let \(M_1\) be a four-dimensional quasitoric manifold over a polygon with \(6k\) vertices, \(k\in \N\).
Moreover, let \(M_2\) be a \(2n\)-dimensional quasitoric manifold over the \(n\)-cube.
Then the orbit polytopes of \(M_1\) and \(M_2\) admit facet colorings with \(2\) and \(n\) colors, respectively.
Therefore the orbit polytope of \(M_1\times M_2\) admits a facet coloring with \(n+2\) colors.
Hence, by Lemma~\ref{sec:twist-dirac-oper-6}, there is a non-vanishing index \(\varphi^c(M_1\times M_2;V,0)\) on \(M_1\times M_2\).
By Lemma~\ref{sec:prod-conn-sums} applied in the case \(V_1=V_2=V\) and \(W_1=W_2=0\), it follows that
\begin{equation*}
  M=(M_1\times M_2)\#(M_1\times M_2)
\end{equation*}
has a non-vanishing index.
Because \(\chi(M)=2\cdot 6k \cdot 2^n-2\) is not divisible by three and four, it follows from Corollary~\ref{sec:twist-dirac-oper-7} and \cite[p. 74-84]{0708.17005} that the only compact simply connected semi-simple Lie-group which can act almost effectively on \(M\) is \(SU(2)\). 
Because connected sums of quasitoric manifolds are quasitoric, the statement follows.
\end{proof}

The connected sum of two quasitoric manifolds is again a quasitoric manifold.
Therefore Lemma~\ref{sec:twist-dirac-oper-6} and the following result may be used to construct more quasitoric manifolds with non-vanishing indices.

\begin{lemma}
\label{sec:prod-conn-sums-1}
  Let \(M_1,M_2\) be quasitoric manifolds of dimension \(2n\geq 4\). Assume that there are vector bundles \(V_1\rightarrow M_1\) and \(W_1\rightarrow M_1\) as in Lemma \ref{sec:prod-conn-sums} and \(b_2(M_2)\leq \dim V_1\) or \(M_2\) is a \(\text{Spin}\)-manifold.

Then there are sums of line bundles \(V,W\) over \(M_1\# M_2\) and a \(\text{Spin}^c\)-structure on \(M_1 \# M_2\) with \(c_1(V)=c_1^c(M_1\# M_2)\), \(c_1(W)\equiv 0 \mod 2\), \(p_1(V+W-TM_1\# M_2)=0\), such that
\begin{equation*}
  \varphi^c(M_1\# M_2;V,W)=2^k\varphi^c(M_1;V_1,W_1)
\end{equation*}
for some \(k\geq 0\).
\end{lemma}
\begin{proof}
  Let \(L_i\rightarrow M_2\), \(i=1,\dots,m\), be line bundles such that the Chern-classes of the \(L_i\) are the Poincar\'e-duals of the characteristic submanifolds of \(M_2\).
Then we have \(TM_2 \oplus \R^{2m-2n}\cong\bigoplus_{i=1}^m L_i\).

We order the \(L_i\) in such a way that
\(c_1(L_1),\dots,c_1(L_{b_2(M_2)})\) form a basis of \(H^2(M_2;\mathbb{Z})\).
Therefore there are \(a_1,\dots,a_{b_2(M)}\in \{0,1\}\) such that
\(c_1(\bigoplus_{i=1}^{b_2(M)} a_i L_i) \equiv c_1(M_2)\mod 2\) and \(W_2=\bigoplus_{i=1}^{b_2(M)} (1-a_i) L_i\oplus \bigoplus_{i=b_2(M)+1}^m L_i\) is a \(\text{Spin}\) bundle.
Consider a \(\text{Spin}^c\)-structure on \(M_2\) such that \(c_1^c(M_2)=c_1(\bigoplus_{i=1}^{b_2(M)} a_i L_i)\).
By Theorem~\ref{sec:twist-dirac-oper-14} we have \(\varphi^c(M_2;V_2,W_2)=0\), where \(V_2=\bigoplus_{i=1}^{b_2(M)} a_i L_i\).
Therefore, by Lemma \ref{sec:prod-conn-sums} the statement follows.
\end{proof}

In dimensions divisible by four we can use Lemma~\ref{sec:prod-conn-sums-1} to improve the results of Corollary \ref{sec:twist-dirac-oper-15}  and prove that there are quasitoric manifolds on which no semi-simple compact Lie-group can act effectively.

\begin{cor}
\label{sec:twist-dirac-oper-13}
  In dimensions \(4k\), \(k>0\), there are infinitely many quasitoric manifolds \(M\) with \(N^{ss}(M)=0\).
\end{cor}
\begin{proof}
  Let \(M'\) be as in Lemma \ref{sec:twist-dirac-oper-6} with \(\dim M'=2n=4k\).
 Then, by an iterated application of Lemma~\ref{sec:prod-conn-sums-1}, there are non-vanishing indices \(\varphi^c(M;V,W)\) on \(M=M'\# l \C P^{2k}\) with \(l\in \mathbb{N}\).
Because connected sums of quasitoric manifolds are quasitoric, \(M\) is quasitoric.

Since a bipartite regular graph has an even number of vertices, it follows from Theorem \ref{sec:twist-dirac-oper-10} that the Euler-characteristic of \(M'\) is even.
Therefore \(\chi(M)=\chi(M')+l\chi(\C P^n) -2l\) is odd if \(l\) is odd.
Because the order of the Weyl-group of a semi-simple compact connected Lie-group is even \cite[p. 74-84]{0708.17005},
the statement follows from Corollary~\ref{sec:twist-dirac-oper-7}.
\end{proof}

\begin{remark}
  Non-singular projective toric varieties are examples of quasitoric manifolds.
  If, in the situation of the proof of Corollary~\ref{sec:twist-dirac-oper-13}, \(M'\) is such a variety, then we can construct infinitely many non-singular toric varieties \(M\) with \(N^{ss}(M)=0\) by blowing-up isolated fixed points in \(M'\) repeatedly, i.e. by taking connected sums with several copies of \(\bar{\C P}^n\).
\end{remark}

\section{Quasitoric manifolds admitting low cohomogeneity actions}
\label{sec:cohom1}

In this section we study quasitoric manifolds which admit a cohomogeneity one or zero action of a compact connected Lie-group and have a non-zero index \(\varphi^c(M;V,W)\).
To do so we need the notion of spaces of \(q\)-type which was introduced by Hauschild \cite{0623.57024}.
A space of \(q\)-type is defined to be a topological space \(X\) satisfying the following cohomological properties:
\begin{itemize}
\item The cohomology ring \(H^*(X;\Q)\) is generated as a \(\Q\)-algebra by elements of degree two, i.e. \(H^*(X;\Q)=\Q[x_1,\dots,x_n]/I_0\) and \(\deg x_i=2\).
\item The defining ideal \(I_0\) contains a definite quadratic form \(Q\).
\end{itemize}

Examples of spaces of \(q\)-type are homogenous spaces of the form \(G/T\) where \(G\) is a semi-simple compact connected Lie-group and \(T\) a maximal torus of \(G\).
Quasitoric manifolds of \(q\)-type were studied in \cite{chapter6}.

For the proof of the main result of this section we need the following lemma.

\begin{lemma}
\label{sec:quas-manif-with}
  Let \(F\rightarrow E\rightarrow B\) be a fibration such that \(\pi_1(B)\) acts trivially on \(H^*(F;\Q)\). If \(F\) and \(B\) are spaces of \(q\)-type then \(E\) is a space of \(q\)-type.
\end{lemma}
\begin{proof}
  Because \(H^*(F;\Q)\) and \(H^*(B;\Q)\) are generated by their degree two parts, it follows from the Serre spectral sequence that \(H^*(E;\Q)\) is generated by its degree two part.
Let \(x_1,\dots,x_m\) be a basis of \(H^2(F;\Q)\) and \(y_1,\dots,y_{m'}\) be a basis of \(H^2(B;\Q)\).
Then there is a basis \(X_1,\dots,X_m,Y_1,\dots,Y_{m'}\) of \(H^2(E;\Q)\) such that \(\iota^*X_i=x_i\), \(i=1,\dots,m\), and \(\pi^*y_i=Y_i\), \(i=1,\dots,m'\).
Here \(\iota:F\rightarrow E\) is the inclusion and \(\pi:E\rightarrow B\) is the projection.

Let \(Q_F\) and \(Q_B\) be positive definite bilinear forms such that \(Q_F(x_1,\dots,x_m)=0\in H^4(F;\Q)\) and  \(Q_B(y_1,\dots,y_{m'})=0\in H^4(B;\Q)\).

Then there are \(\alpha_{11},\dots,\alpha_{mm'}\in \Q\) and \(\beta_1,\dots,\beta_{m'}\in \Q\) such that for all \(\lambda\in \Q\):
\begin{align*}
  Q_\lambda(X_1,\dots,X_m,Y_1,\dots,Y_{m'})&= Q_F(X_1,\dots,X_m) + \lambda Q_B(Y_1,\dots,Y_{m'})\\ & + \sum_{i,j} \alpha_{ij} X_iY_j+ \sum_i \beta_i Y_i^2 \\ &=0\in H^4(E;\Q).
\end{align*}

We claim that \(Q_\lambda\) is positive definite for sufficient large \(\lambda\).
To see that it is sufficient to show that for all \(a\in S^{m+m'-1}\subset \R^{m+m'}\), \(Q_\lambda(a)>0\).
We may write \(a=\gamma_1 x + \gamma_2 y\) with \(x\in \R^m\), \(y\in \R^{m'}\), \(\|x\|=\|y\|=1\) and \(\gamma_1^2+\gamma_2^2=1\).

Because \(Q_F\) is positive definite and  \(S^{m+m'-1}\cap \R^m\) is compact, there is an \(\epsilon >0\) such that \(Q_\lambda(a) - \lambda Q_B(a)>0\) for all \(\gamma_2<\epsilon\).

Because \(Q_B\) is positive definite and \(S^{m+m'-1}\cap \{\gamma_2\geq \epsilon\}\) is compact, we may take \(\lambda\) sufficiently large such that
\begin{align*}
  \lambda \min \{Q_B(a);\;a\in S^{m+m'-1}\cap \{\gamma_2\geq \epsilon\}\}& > -
 \min \{Q_\lambda(a)-\lambda Q_B(a);\\
 &\quad\quad a\in S^{m+m'-1}\cap \{\gamma_2\geq \epsilon\}\}.
\end{align*}
Therefore \(Q_\lambda\) is positive definite for sufficient large \(\lambda\).
\end{proof}

Now we can prove the following theorem.

\begin{theorem}
\label{sec:quas-manif-with-1}
  Let \(M\) be a quasitoric manifold on which a compact connected Lie-group \(G\)  acts  such that \(\dim M/G\leq 1\).
  Assume that \(M\) has a non-vanishing index \(\varphi^c(M;V,W)\) with \(V\), \(W\) as in Theorem~\ref{sec:twist-dirac-oper-8}.

 Then \(M=\prod S^2\) if \(\dim M/G=0\) or \(M\) is a \(S^2\)-bundle with structure group a maximal torus of \(G\) over \(\prod S^2\) if \(\dim M/G=1\).
\end{theorem}
\begin{proof}
  If \(\dim M/G=0\) then \(M\) is a homogeneous space \(G/H\).
  Because \(\chi(M)\neq 0\), \(H\) must have maximal rank in \(G\).
  Therefore we may assume that \(G\) is semi-simple.
  Hence, it follows from Corollary~\ref{sec:twist-dirac-oper-5} that \(H\) is a maximal torus of \(G\).
  As in section 3 of \cite{chapter6}, one sees that \(M=\prod S^2\).

  Now assume that \(\dim M/G=1\).
Because \(\chi(M)\neq 0\), it follows from Corollary~\ref{sec:twist-dirac-oper-5} that there is an orbit of type \(G/T\) with \(T\) a maximal torus of \(G\).
 Because \(\dim G/T\) is even this must be a non-principal orbit.
Hence, the orbit space \(M/G\) is homeomorphic to the compact interval \([0,1]\) and there is exactly one other non-principal orbit.
Let \(S\subset G\) be a principal isotropy group.
Then we may assume \(S\subset T\).
Moreover, \(T/S\) is a sphere.
Therefore \(S\) has codimension one in \(T\).

Let \(K^+\) be the isotropy group of the other non-principal orbit.
Then \(K^+/S\) is a sphere and the identity component of \(K^+\) is a torus by Theorem \ref{sec:twist-dirac-oper-8}.
Therefore there are two cases
\begin{itemize}
\item \(\dim K^+=\dim S\) and \(K^+/S=\mathbb{Z}_2\).
\item \(K^+\) is a maximal torus of \(G\).
\end{itemize}

In the first case, we have by Seifert-van Kampen's theorem
\begin{equation*}
  \pi_1(M)= \pi_1(G/T)*_{\pi_1(G/S)}\pi_1(G/K^+)= \pi_1(G/K^+)/\pi_1(G/S)=\mathbb{Z}_2
\end{equation*}
because \(G/S\rightarrow G/K^+\) is a twofold covering.
But \(M\) is simply connected because it is quasitoric.
So this case does not occur.

Now as in the remark before Lemma 5.2 of \cite{chapter6} one sees that \(M\) is a \(S^2\)-bundle with structure group \(T\) over \(G/T\).
By Lemma~\ref{sec:quas-manif-with}, \(M\) is a quasitoric manifold which is of \(q\)-type.
Therefore it follows from Theorem 5.3 of \cite{chapter6} that \(M\) is a \(S^2\)-bundle over \(\prod S^2\).
\end{proof}

\section{Quasitoric manifolds with non-vanishing indices and $N(M)\geq 3n-4$}
\label{sec:cube}

By Corollary~\ref{sec:twist-dirac-oper-11} the symmetry degree of a quasitoric manifold \(M\) with a non-trivial index \(\varphi^c(M;V,W)\) is bounded from above by \(3n\).
In this section we classify those \(2n\)-dimensional quasitoric manifolds which admit a twisted Dirac-operator with a non-vanishing index and have degree of symmetry greater or equal to \(3n-4\).

For the statement of our first theorem we need the notion of a torus manifold.
A torus manifold is a \(2n\)-dimensional closed connected orientable smooth manifold \(M\) with an effective smooth action of an \(n\)-dimensional torus \(T\), such that \(M^T\) is non-empty.

\begin{theorem}
\label{sec:quas-manif-over}
  Let \(M\) be a \(2n\)-dimensional quasitoric manifold with non-vanishing index \(\varphi^c(M;V,W)\) with \(V\), \(W\) as in Theorem~\ref{sec:twist-dirac-oper-8} and \(G\) be a compact connected Lie-group of rank \(n\), which acts almost effectively on \(M\).
  Then \(G\) has a covering group of the form \(\prod SU(2)\times T^{l_0}\).
  Moreover, \(M\) is a fiber bundle with fiber a \(2l_0\)-dimensional torus manifold over \(\prod S^2\).
\end{theorem}
\begin{proof}
  \(M\) is a torus manifold with \(G\)-action in the sense of \cite{torus} and \(H^*(M;\mathbb{Z})\) is generated by \(H^2(M;\mathbb{Z})\). Therefore \(G\) has a covering group of the form  \(\tilde{G}=\prod_i SU(l_i+1)\times T^{l_0}\) by Remark 2.9 of \cite{torus}.
  Let \(T\) be a maximal torus of \(G\) and \(x\in M^T\). Then, by Lemmas 3.1 and 3.4 of \cite{torus}, we have
  \begin{align*}
    SU(l_i+1)_{x}&=S(U(l_i)\times U(1)) & \text{or}&& SU(l_i+1)_{x}&=SU(l_i+1).
  \end{align*}
  Therefore by Corollary~\ref{sec:twist-dirac-oper-5}, we have \(l_i=1\).
  Moreover, each factor \(SU(l_i+1)\) does not have a fixed point in \(M\).
  Therefore the second statement follows from an iterated application of Corollary 5.6 of \cite{torus}.
\end{proof}

The next theorem is the classification announced in the introduction of this section.

\begin{theorem}
\label{sec:quas-manif-with-2}
Let \(M\) be a \(2n\)-dimensional quasitoric manifold with non-vanishing index \(\varphi^c(M;V,W)\) with \(V\), \(W\) as in Theorem~\ref{sec:twist-dirac-oper-8}. If \(N(M)\geq 3n-4\), then \(M\) is diffeomorphic to one of the manifolds in the following list.

\begin{center}
\begin{tabular}{|c|c|}
  \(N(M)\) & \(M\)\\\hline\hline
  \(3n\) & \(\prod S^2\)\\\hline
  \(3n-1\)& impossible\\\hline
  \(3n-2\)& \(S^2\)-bundle over \(\prod S^2\)\\\hline
  \(3n-3\)& impossible\\\hline
  \(3n-4\)& \(N\)-bundle over \(\prod S^2\) with \(N\) a quasitoric manifold, \(\dim N=4\) \\
\end{tabular}  
\end{center}
\end{theorem}
\begin{proof}
The statement about the quasitoric manifolds with \(N(M)=3n\) follows from Corollary~\ref{sec:twist-dirac-oper-11}.

Therefore assume that \(M\) is a \(2n\)-dimensional quasitoric manifold with non-vanishing index \(\varphi^c(M;V,W)\) and \(G\) is a compact connected Lie-group of dimension \(3n-1\), which acts effectively on \(M\).
Let \(T\) be a maximal torus of \(G\).

Because, by Corollary~\ref{sec:twist-dirac-oper-3},
\begin{equation}
  \label{eq:2}
  \dim G -\dim T \leq 2n,
\end{equation}
we have \(\dim T = n-1\) or \(\dim T=n\).
If \(\dim T=n\), then \(\dim G-\dim T\) is odd,  which is impossible.
But \(\dim T=n-1\) is impossible by Corollary~\ref{sec:twist-dirac-oper-3}.

Let \(M\), \(G\), \(T\) as above. But with \(\dim G=3n-2\).
By (\ref{eq:2}), we have \(\dim T=n-2,n-1,n\).
As in the first case one sees that \(\dim T=n-2,n-1\) are impossible.
If \(\dim T=n\), we see with Theorem~\ref{sec:quas-manif-over} that \(M\) is a \(S^2\)-bundle over \(\prod S^2\).

Let \(M\), \(G\), \(T\) as above. But with \(\dim G = 3n-3\).
By (\ref{eq:2}), we have \(\dim T= n-3,n-2,n-1,n\).
As above one sees that \(\dim T = n-3,n-2,n\) are impossible.
Therefore we have \(\dim T=n-1\).
Because \(\chi(M)\neq 0\), there is by Corollary~\ref{sec:twist-dirac-oper-5} an orbit of type \(G/T\) which has dimension \(2n-2\).

Therefore the principal orbit type has dimension \(2n-2\) or \(2n-1\).
In the first case the principal orbit type is \(G/T\) and by Corollary~\ref{sec:twist-dirac-oper-5} there is no exceptional or singular orbit.
Hence, \(M\) is a fiber bundle over a simply connected surface with fiber \(G/T\) and structure group \(N_GT/T\).
Since \(N_GT/T\) is finite, we have \(M=S^2\times G/T\).
Therefore we have \(N(M)\geq 3 + \dim G =3n\).

Now assume that the principal orbit \(G/S\) has codimension one in \(M\).
Then, by Theorem~\ref{sec:quas-manif-with-1}, \(M\) is a \(S^2\)-bundle with structure group a torus over \(\prod_{i=1}^{n-1} S^2\).
Therefore we have \(N(M)\geq 3n-2\).

Now let \(M\), \(G\), \(T\) as above, but with \(\dim G=3n-4\).
By (\ref{eq:2}), we have \(\dim T= n-4,n-3,n-2,n-1,n\).
As above one sees that \(\dim T= n-4,n-3,n-1\) are impossible.
Therefore we have \(\dim T=n-2,n\).

At first assume that \(\dim T=n\).
Then \(M\) is a torus manifold with \(G\)-action.
By Theorem~\ref{sec:quas-manif-over} we have \(G=\prod_{i=1}^k SU(2)\times T^{l_0}\) with \(3n-4 = 3k+ l_0\) and \(n=k+l_0\).
Therefore we have \(l_0=2\) and \(M\) is a fiber bundle with fiber a four-dimensional torus manifold over \(\prod_{i=1}^kS^2\).
By Lemma 5.17 of \cite{torus}, the fiber of this bundle is simply connected.
Therefore it is quasitoric because every four-dimensional simply connected torus manifold is quasitoric \cite[Section 5]{0216.20202}.

Now assume that \(\dim T = n-2\).
Then we have \(\dim G/T=2n-2\).
Therefore the principal orbit type of the \(G\)-action on \(M\) has dimension \(2n-2\) or \(2n-1\).
In both cases one sees as in the case \(\dim G= 3n-3\) that \(N(M)\geq 3n-2\).
\end{proof}

\section{Highly symmetric quasitoric manifolds}
\label{sec:highly}

In this section we show that \(\C P^n\) is the most symmetric quasitoric manifold of dimension \(2n\).
The main result of this section is the following theorem.

\begin{theorem}
\label{sec:highly-symm-quas-7}
  Let \(M\) be a \(2n\)-dimensional quasitoric manifold. Then we have
  \begin{equation*}
    N(M)\leq n^2+2n
  \end{equation*}
with equality only holding for \(M=\C P^n\).
\end{theorem}

The proof of this theorem is subdivided into several lemmas.
We prove it separately in each dimension.
We begin with dimensions \(2n\geq 20\).

\begin{lemma}
\label{sec:highly-symm-quas-5}
  Let \(M\) be a quasitoric manifold of dimension \(2n\geq 20\) with \(M\neq \C P^n\). Then we have \(N(M)\leq n^2+n+1<n^2+2n=N(\C P^n)\).
\end{lemma}
\begin{proof}
  It was shown by Ku, Mann, Sicks and Su \cite[Theorem 1, p.141]{0191.54401} that if \(H^\alpha(M;\Q)\neq 0\) and \(M\neq \C P^n\), then
  \begin{equation*}
    N(M)\leq \frac{\alpha(\alpha+1)}{2} + \frac{(2n-\alpha)(2n-\alpha +1)}{2}.
  \end{equation*}
 The statement follows from this result applied in the cases \(\alpha =n\) or \(\alpha=n-1\).
\end{proof}

Now we turn to the low dimensional case \(2n\leq 8\).

\begin{lemma}
\label{sec:highly-symm-quas-4}
Let \(M\) be a quasitoric manifold of dimension \(2n\), \(n\leq 4\), and \(G\) a compact connected Lie-group which acts almost effectively on \(M\).
Then \(\dim G \leq n^2+2n\) and equality only holds for \(M=\C P^n\) and \(\tilde{G}=SU(n+1)\).
\end{lemma}
\begin{proof}
  Because \(M\) has non-zero Euler-characteristic, we have \(\rank G \leq n\).
  If \(\rank G=n\), it follows from Remark 2.9 of \cite{torus} that \(G\) has a covering group of the form \(\tilde{G}=\prod_{i=1}^k SU(l_i+1)\times T^{l_0}\) with \(\sum_{i=0}^k l_i =n\).
  Therefore we have \(\dim G\leq n^2+2n\) with equality holding if and only if \(\tilde{G}=SU(n+1)\).
  In the latter case it follows from Corollary 8.9 of \cite{torus} that \(M=\C P^n\).

  Now assume that \(\rank G\leq n-1\).
  The highest dimensional Lie-groups of rank \(k\) are as follows
  \begin{center}
    \begin{tabular}{|c|c|c|c|}
      \(k\)&\(G\)&\(\dim G\)&\((k+1)^2+2(k+1)\)\\\hline\hline
      \(1\)&\(\text{Spin}(3)\) & \(3\)& \(4\)\\\hline
      \(2\)&\(G_2\)&\(14\)&\(15\)\\\hline
      \(3\)&\(\text{Spin}(7)\)&\(21\)&\(24\)
    \end{tabular}
  \end{center}
Therefore the statement follows.
\end{proof}

Now we turn to the middle dimensions \(10\leq 2n\leq 18\).
Those \(2n\)-dimensional simply connected manifolds on which compact connected non-abelian Lie-groups of rank \(n\) act were classified in \cite{torus}.
Therefore we first focus on actions of those groups which have a rank which is smaller than \(n\).

As a first step we show that if a high-dimensional Lie-group acts on a quasitoric manifolds of these dimensions, then its simply connected covering group has a big simple factor which is isomorphic to \(\text{Spin}(k)\).

\begin{lemma}
\label{sec:highly-symm-quas}
  Let \(M\) be a manifold of dimension \(2n\), \(5\leq n\leq 9\), and \(G\) a compact connected Lie-group with \(\rank G \leq n-1\) and \(\dim G \geq n^2+2n\) which acts almost effectively on \(M\). Then \(G\) has a covering group of the form
  \begin{equation*}
    \tilde{G}=\text{Spin}(k)\times G'
  \end{equation*}
with \(k=9\) if \(n=5\) and
\begin{equation*}
  k\geq
  \begin{cases}
    11 &\text{if } n=6\\
    12 &\text{if } n=7\\
    13 &\text{if } n=8\\
    15 &\text{if } n=9.
  \end{cases}
\end{equation*}
\end{lemma}
\begin{proof}
  Let \(G/H\) be a principal orbit type of the \(G\)-action on \(M\).
Then
\begin{equation*}
  \dim G \geq n^2+2n\geq (\frac{n}{2}+1)\dim M \geq(\frac{n}{2}+1)\dim G/H
\end{equation*}
Because \(\frac{n}{2}+1\geq \frac{14}{4}\), we may apply Proposition B of \cite[p. 135]{0191.54401} with \(r=\frac{n}{2}+1-\epsilon\).
Therefore \(\tilde{G}\) is of the form
\begin{align*}
  &\text{Spin}(k)\times G'& k&\geq n+2, \text{ or}\\
  &SU(k)\times G'&k&\geq n+1,\text{ or}\\
  &Sp(k)\times G'&k&\geq n.
\end{align*}
Because \(\rank G \leq n-1\), the last two cases do not occur.

It remains to prove that the lower bound for \(k\) given in the lemma holds.
This follows from an inspection of the dimensions of those groups which have \(\text{Spin}(k)\), \(k\geq n+2\), as a simple factor and rank bounded from above by \(n-1\).
These groups are listed in the following tables.
Here we have omitted those groups which are not isomorphic to \(\text{Spin}(k)\) and for which the \(\text{Spin}(k)\)-factor alone has a dimension greater or equal to \(n^2+2n\).
If the \(\text{Spin}(k)\)-factor has a lower dimension, we have only listed those groups which have maximal dimension among those groups which have this \(\text{Spin}(k)\)-factor.

If \(n=5\) we have
  \begin{center}
    \begin{tabular}{|c|c|c|}
      \(G\)&\(\dim G\)&\(n^2+2n=35\)\\\hline\hline
      \(\text{Spin}(9)\)&\(36\)&\\\hline
      \(\text{Spin}(8)\)&\(28\)&\\\hline
      \(\text{Spin}(3)\times \text{Spin}(7)\)&\(24\)&
    \end{tabular}
  \end{center}

  For \(n=6\) we have
  \begin{center}
    \begin{tabular}{|c|c|c|}
      \(G\)&\(\dim G\)&\(n^2+2n=48\)\\\hline\hline
      \(\text{Spin}(11)\)&\(55\)&\\\hline
      \(\text{Spin}(10)\)&\(45\)&\\\hline
      \(\text{Spin}(3)\times \text{Spin}(9)\)&\(39\)&\\
    \end{tabular}
  \end{center}

For \(n=7\) we have
  \begin{center}
    \begin{tabular}{|c|c|c|}
      \(G\)&\(\dim G\)&\(n^2+2n=63\)\\\hline\hline
      \(\text{Spin}(13)\)&\(78\)&\\\hline
      \(\text{Spin}(12)\)&\(66\)&\\\hline
      \(\text{Spin}(3)\times \text{Spin}(11)\)&\(58\)&\\\hline
      \(G_2\times \text{Spin}(9)\)&\(50\)&
    \end{tabular}
  \end{center}

For \(n=8\) we have
  \begin{center}
    \begin{tabular}{|c|c|c|}
      \(G\)&\(\dim G\)&\(n^2+2n=80\)\\\hline\hline
      \(\text{Spin}(15)\)&\(105\)&\\\hline
      \(\text{Spin}(14)\)&\(91\)&\\\hline
      \(\text{Spin}(3)\times \text{Spin}(13)\)&\(81\)&\\\hline
      \(\text{Spin}(3)\times \text{Spin}(12)\)&\(69\)&\\\hline
      \(G_2\times \text{Spin}(11)\)&\(69\)&
    \end{tabular}
  \end{center}

For \(n=9\) we have
  \begin{center}
    \begin{tabular}{|c|c|c|}
      \(G\)&\(\dim G\)&\(n^2+2n=99\)\\\hline\hline
      \(\text{Spin}(17)\)&\(136\)&\\\hline
      \(\text{Spin}(16)\)&\(120\)&\\\hline
      \(\text{Spin}(15)\)&\(105\)&\\\hline
      \(\text{Spin}(3)\times \text{Spin}(14)\)&\(94\)&\\\hline
      \(G_2\times \text{Spin}(13)\)&\(92\)&\\\hline
      \(\text{Spin}(7)\times\text{Spin}(11)\)&\(69\)&\\
    \end{tabular}
  \end{center}
Therefore the statement about \(k\) follows.
\end{proof}

The next step is to identify the identity component of the principal isotropy group of the \(\text{Spin}(k)\)-action on \(M\).

\begin{lemma}
\label{sec:highly-symm-quas-3}
  Let \(M\), \(G\) as in Lemma~\ref{sec:highly-symm-quas}.
  If \(n=5\), then also assume that \(\chi(M)\neq 0\).
  Then the identity component of the principal isotropy group of the \(\text{Spin}(k)\)-action on \(M\) is \(\text{Spin}(k-1)\).
\end{lemma}
\begin{proof}
  If \(6\leq n\leq 9\), one can argue as in the proof of the main lemma of \cite[p.135]{0191.54401} in case III.

  Therefore assume that \(n=5\).
  Then we have \(k=9\).
Because \(\chi(M)\neq 0\), there is a point \(x\in M\) such that \(\text{Spin}(9)_x\) has maximal rank in \(\text{Spin(9)}\).
By the classification of maximal rank subgroups of \(\text{Spin}(9)\) in \cite{0034.30701} and the dimension assumption, it follows that \(\text{Spin}(9)_x^0=\text{Spin}(8)\) or \(\text{Spin}(9)_x^0=\text{Spin}(9)\).

If \(\text{Spin(9)}_x^0=\text{Spin}(8)\), then the orbit of \(x\) has codimension two in \(M\). Because \(\text{Spin}(8)\) has no non-trivial \(2\)-dimensional representation, it follows that \(\text{Spin}(8)\) is the identity component of a principal isotropy group.

If \(\text{Spin}(9)_x^0=\text{Spin}(9)\), then \(T_xM\) is a \(10\)-dimensional representation of \(\text{Spin}(9)\).
Therefore it is the sum of the standard \(9\)-dimensional representation of \(\text{Spin}(9)\) and the trivial one dimensional representation.
Hence, the statement follows in this case.
\end{proof}

As a consequence of Lemmas~ \ref{sec:highly-symm-quas} and \ref{sec:highly-symm-quas-3} we get the following lemma which implies Theorem~\ref{sec:highly-symm-quas-7} in the remaining dimensions. 

\begin{lemma}
\label{sec:highly-symm-quas-6}
  Let \(M\) be a quasitoric manifold of dimension \(2n\), \(5\leq n\leq 9\), and \(G\) be a compact connected Lie-group which acts almost effectively on \(M\).

Then \(\dim G\leq n^2+2n\) and equality only holds for \(M=\C P^n\) and \(\tilde{G}=SU(n+1)\).
\end{lemma}
\begin{proof}
  Since \(M\) has non-zero Euler-characteristic, we have \(\rank G\leq n\).
  In the case \(\rank G=n\), one can argue as in the proof of Lemma~\ref{sec:highly-symm-quas-4}.

  Therefore we may assume that \(\rank G\leq n-1\).
  Assume that \(\dim G\geq n^2+2n\).
  By Lemmas  \ref{sec:highly-symm-quas} and \ref{sec:highly-symm-quas-3}, there is an almost effective action of \(\text{Spin}(k)\) on \(M\) such that \(\dim M/\text{Spin}(k)\leq 4\) and all orbits are acyclic over \(\Q\) up to dimension \(7\). 
By the Vietoris-Begle-mapping theorem, it follows that
\begin{equation*}
  0\neq H^6(M;\Q)\cong H^6(M/\text{Spin}(k);\Q)=0.
\end{equation*}
This is a contradiction.
\end{proof}

\appendix

\section{Groups acting on tori}
\label{sec:tori}

In this appendix we prove some of the technical details which are needed in the proof of Lemma~\ref{sec:twist-dirac-oper-4}.

\begin{lemma}
\label{sec:groups-acting-tori}
  Let \(M\) be a free, finitely generated \(\mathbb{Z}\)-module and \(G\) a finite group which acts on \(M\).
Then there is a \(G\)-invariant submodule \(M'\subset M\) such that:
\begin{enumerate}
\item \(M'\cap M^G=\{0\}\)
\item \(\rank M' + \rank M^G = \rank M\)
\end{enumerate}
\end{lemma}
\begin{proof}
  Choose a positive definite \(G\)-invariant metric on \(M\).
Then the orthogonal complement \(M'\) of \(M^G\) is \(G\)-invariant and \(M'\cap M^G=\{0\}\).
Moreover, we have \(\rank M' + \rank M^G=\rank M\).
\end{proof}

\begin{lemma}
\label{sec:groups-acting-tori-1}
  Let \(G\) be a finite group, which acts by automorphisms on the torus \(T\). Then there are subtori \(T_1,T_2\subset T\) such that:
  \begin{enumerate}
  \item \(T_1\subset T^G\)
  \item \(T_2^G\) is finite
  \item \(\langle T_1,T_2\rangle =T\)
  \end{enumerate}
\end{lemma}
\begin{proof}
  The action of \(G\) on \(T\) induces an action of \(G\) on the Lie-algebra \(LT\) of \(T\). 
Let \(M\) be the integer lattice in \(LT\).
Then \(M\) is \(G\)-invariant. 
Let \(M'\) and \(M^G\) as in Lemma~\ref{sec:groups-acting-tori} and \(T_2\) be the subtorus of \(T\) corresponding to \(M'\); \(T_1\) the subtorus of \(T\) corresponding to \(M^G\).

Then we have:
\begin{itemize}
\item \(T_1\subset T^G\) because \(G\) acts trivially on the Lie-algebra of \(T_1\).
\item \(T_2^G\) is finite because of (1) in Lemma~\ref{sec:groups-acting-tori}.
\item \(\langle T_1,T_2\rangle=T\) because of (1) and (2) in Lemma~\ref{sec:groups-acting-tori}.
\end{itemize}
\end{proof}

\bibliography{dirac}{}
\bibliographystyle{amsplain}
\end{document}